\documentclass[11pt, reqno]{amsart}
\usepackage[utf8]{inputenc}
\usepackage{amsmath,amsthm,amsfonts,amssymb,mathrsfs,stmaryrd,bigints}
\usepackage{dsfont}
\usepackage{mathtools}
\usepackage[usenames]{color}

\usepackage[colorlinks=true,linkcolor=blue]{hyperref}
\newtheorem{theorem}{Theorem}[section]
\newtheorem{lemma}[theorem]{Lemma}
\newtheorem{proposition}[theorem]{Proposition}

\newtheorem{remark}[theorem]{Remark}
\newtheorem{maintheorem}{Theorem}

\usepackage{epstopdf} 
\usepackage{amssymb}
\usepackage{setspace}
\usepackage{enumerate}
\usepackage{bigstrut}
\usepackage{multirow}
\usepackage{mathtools,xparse}
\usepackage{mathrsfs}
\usepackage{hyperref}
\usepackage{xcolor}
\usepackage [autostyle, english = american]{csquotes}

\DeclareMathOperator*{\argmax}{arg\,max}
\allowdisplaybreaks
\usepackage{chngcntr}

\def\it{\textit}

\newcommand{\E}{{\mathbb{E}}}
\newcommand{\1}{\mathds{1}}

\newcommand{\tmix}{t_{\rm{mix}}}
\newcommand{\R}{\mathbb{R}}
\newcommand{\norm}[1]{\left\lVert#1\right\rVert}

\newcommand{\e}{\varepsilon}

	\renewcommand{\P}{\mathbb{P}}

\newcommand{\cE}{\mathcal{E}}

\newcommand{\cG}{\mathcal{G}}

\newcommand{\cL}{\mathcal{L}}

\newcommand{\cP}{\mathcal{P}}

\newcommand{\cT}{\mathcal{T}}

\newcommand{\sP}{\mathscr{P}}

  \renewcommand{\vec}[1]{\boldsymbol{#1}}

\renewcommand{\setminus}{\backslash}

\begin{document}

\title[Sub-critical Exponential random graphs]{Sub-critical Exponential random graphs: concentration of measure and some applications}

\author{Shirshendu Ganguly and Kyeongsik Nam}

\begin{abstract}
The exponential random graph model (ERGM) is a central object in the study of clustering properties in social networks as well as canonical ensembles in statistical physics.
Despite some breakthrough works in the mathematical understanding of ERGM, most notably in \cite{bbs}, through the analysis of a natural Heat-bath Glauber dynamics and in \cite{cd, e, eg}, via a large deviation theoretic perspective, several basic questions have remained unanswered owing to the lack of exact solvability unlike the much studied Curie-Weiss model (Ising model on the complete graph). In this paper, we  establish a series of new concentration of measure results for the ERGM \emph{throughout the entire sub-critical phase}, including a Poincar\'e inequality, Gaussian concentration for Lipschitz functions, and a central limit theorem.  In addition,  a new proof of a quantitative bound on the $W_1-$Wasserstein distance to Erd\H{o}s-R\'enyi graphs, previously obtained in \cite{rr}, is also presented. The arguments rely on translating temporal mixing properties of Glauber dynamics to static spatial mixing properties of the equilibrium measure and have the potential of being useful in proving similar functional inequalities for other Gibbsian systems beyond the perturbative regime.
\end{abstract}

\address{ Department of Statistics, Evans Hall, University of California, Berkeley, CA
94720, USA} 

\email{sganguly@berkeley.edu }

\address{ Department of Mathematics, University of California, Los Angeles, CA
90095, USA} 

\email{ksnam@math.ucla.edu }


 
\maketitle

\tableofcontents

\section{Introduction}
A central object in the study of statistical models on networks, is the notion of a Gibbs measure on graphs. In the most general setting, the probability of a graph $G$ on $n$ vertices, thought of naturally as $x\in  \{0,1\}^{n(n-1)/2}$ is proportional to $e^{\beta f(x)}$ where $f(\cdot)$ is the Hamiltonian and $\beta$ is a parameter classically referred to as the inverse temperature. 

A particularly important subclass of such measures, capturing clustering properties, is obtained when the Hamiltonian is given by counts of subgraphs of interest, such as triangles. This is termed in the literature as the \textit{Exponential Random Graph model} (ERGM).  Thus more precisely, for $x\in \{0,1\}^{n(n-1)/2}$, where the configuration space is the set of all graphs on the vertex set $\{1,\cdots,n\}$, defining $N_G(x)$ {to} be the number of labeled subgraphs $G$ in $x$, given a vector $\vec \beta=(\beta_1,\ldots, \beta_s),$ the ERGM Gibbs measure is defined as
\begin{align} \label{gibbs}
{\pi}(x) \sim \exp \Big(\sum_{i=1}^s \beta_i \frac{N_{G_i}(x)}{n^{|V(G_i)|-2}}\Big),
\end{align} 
(see Section \ref{section 1.2} for details)
and hence is a version of the well known Erd\H{o}s-R\'enyi graphs, obtained by \it{tilting} according to the subgraph counting Hamiltonian. Being mostly used for modeling relational networks in sociology, there is a significant amount of rigorous and non-rigorous literature on ERGM, see e.g. \cite{fs,hl, pn1, pn2} for some specific cases, while \cite{cd2} verified a mean-field
approximation for some values of $\vec \beta,$ where the ERGM behaves qualitatively like an Erd\H{o}s-R\'enyi graph. There is also a series of works studying constrained ERGM models, as an important example of micro-canonical ensembles in statistical physics (see e.g. \cite{kenyon} and the references therein). 
A characteristic property of the ERGM is the well known mean field behavior, which informally means that it can be approximated in an information theoretic sense by product measures. 
Unfortunately, contrary to classical spin systems and lattice gases, a detailed analysis of  general  ERGM has been out of reach so far. Thus, while a lot of refined results on fluctuation theory and concentration properties, have been established over the years, for the exactly solvable Curie-Weiss model, (Ising model on the complete graph), corresponding questions for the ERGM remain largely open. 

However a significant breakthrough was made in \cite{bbs}, who studied a natural Heat-bath Glauber dynamics on $\cG_n$ with ferromagnetic ERGM as the invariant measure, and established precise estimates on convergence to equilibrium as well as closeness to an Erd\H{o}s-R\'enyi measure of appropriate density as a function of $\vec \beta.$

Soon after, in a landmark development Chatterjee and Diaconis \cite{cd} introduced a large deviation theoretic approach to the study of ERGM. Among many things, one of the key achievements of this work, is a variational formula for the free energy of ERGM, using the large deviation theory for Erd\H{o}s-R\'enyi graphs, developed in prior work \cite{cv} relying on the theory of graph limits developed by Lov\'asz and coauthors \cite{lov}.  Analyzing the formula, it was also established that ERGM for certain values of $\beta$ behaves qualitatively like an Erd\H{o}s-R\'enyi graph (in an entropy theoretic sense) in the thermodynamic limit. More recently, a more refined result was established by Eldan \cite{e}, who obtained a representation of the ERGM as a low entropy mixture of product measures, using the framework of nonlinear large deviation theory developed in the same paper, extending the theory put forward in \cite{CD1}. 

Recently, the results in \cite{bbs} have been extended to what are known as vertex weighted exponential random graphs in \cite{dey}. We refer to \cite{eg,ry,sr,yz} for more results on ERGM and \cite{c2} for a beautiful exposition of the recent developments around the general theory of large deviations for random graphs.

In spite of the above developments, the understanding of the ERGM was still not refined enough to treat important and delicate questions about concentration of measure properties and central limit theorems with the exception of the perturbative, very high temperature regime, popularly known as \it{Dobrushin's uniqueness (DU) regime} where \cite{ss} showed that the ERGM behaves qualitatively like a product measure, satisfying the Log-Sobolev inequality (LSI) and as a consequence, strong concentration properties.

However, going beyond perturbative ideas, while for classical spin systems and related percolation models, there has been significant progress in understanding probabilistic aspects in the various phases of temperature, analogous questions for non-exactly solvable mean-field models like the ERGM, about spectral gap, Log-Sobolev inequality, concentration of measure, and central limit theorems are largely open.

To elaborate on a concrete difficulty one faces, let us, for example, consider the concentration phenomenon for the Ising model on the lattice, which can be obtained  by a martingale argument or coercivity properties such as Poincar\'e and Log-Sobolev inequality. The validity of these conditions require a  spatial mixing property which has been verified in some cases throughout the entire high temperature phase, on the lattice (see \cite{bkmp} for the related results for the Ising model on general graphs).  In fact, it is known that  a certain  mixing condition, well known as the  \it{Dobrushin-Shlosman mixing condition}, is equivalent to the uniform boundedness of the Log-Sobolev constant (see \cite{mo1,mo2,sz0,sz} for details). However, while spatial mixing properties have been an object  of intense study for finite range spin systems, such properties cannot be expected for the ERGM owing to the natural symmetry, exchangeability and the mean field nature of the model. 

In this paper, we establish a series of new results for the ERGM which hold throughout the \it{whole high temperature regime}. The main result establishes  Gaussian concentration for Lipschitz functions, and as applications we prove a central limit theorem (CLT) for a partial number of edges as well as a quantitative result about how close the high temperature ERGM is to an Erd\H{o}s-R\'enyi graph. We also establish a sharp Poincar\'e  inequality.  Most of the results in the paper are the first of their kind beyond the perturbative regime for the ERGM. We also include a discussion about potential applications of some of the results in this paper to problems in statistics about estimability of parameters in an ERGM, which along with the CLT result, can lead to progress in the analysis of natural hypothesis testing problems on networks. 

The key ingredient we rely on is a temporal mixing result along the Heat-bath Glauber dynamics derived from \cite{bbs}. Thus the main theme in this paper is how such estimates can be translated into mixing properties of the equilibrium measure which we expect would be of general interest particularly in analyzing other related Gibbs measures.\\

We now move on to the precise forms of the main results and start by developing the necessary notation.
\subsection{Configuration space and notations}
For any graph $G$, we will use $V(G)$ and $E(G)$ to denote the vertex and edge sets, respectively. 
Let $\mathcal{G}_n$ be the set of all graphs with  vertex set $[n]:=\{1,2,\cdots,n\}$. In particular, let $K_n$ be a complete graph with a vertex set $[n]$. Adopting a widely used notation, an element in $\mathcal{G}_n$ will be denoted by $x=(x_e)_{e\in E(K_n)}$ with  $x_e=1$ if the (undirected) edge $e=(i,j)$ is present and $x_e=0$ otherwise, i.e., we will identify  $\mathcal{G}_n$ with the hypercube $\{0,1\}^M$, where $M:=\frac{n(n-1)}{2}$.

For $x\in \mathcal{G}_n$, let $E(x)$  be the set of edges $e$ with $x_e=1$. Define $\bar{x}_e=(x_f)_{f\neq e}$  to be the collection of all coordinates of $x$ except for the edge $e$. Also, denote $x_{e+}$ and $x_{e-}$ to be configurations whose edge sets are $E(x) \cup e$ and $E(x) \setminus e$, respectively. {Further, let $x_{ef,+}$ be a configuration whose edge set is $E(x)\cup e\cup f$.}  Finally, for any function $\varphi:\mathcal{G}_n \rightarrow \R$, let 
\begin{align}\label{discretederivatives1}
\partial_e \varphi(x) &:=\varphi(x_{e+})- \varphi(x_{e-}),\\ 
\nonumber
\partial_{ef} \varphi(x)&:= \partial_e \varphi(x_{f+})-\partial_e \varphi(x_{f-})
\end{align}
denote the discrete first and second derivatives.

We now define a natural partial ordering and metric on $\mathcal{G}_n$.  We say $x\leq y,$ if and only if $x_e\leq y_e$ for all edges $e$. For $x,y\in \mathcal{G}_n$, define $x \wedge y:=(x_e \wedge y_e)_{e\in E(K_n)}$ with $x_e \wedge y_e = \min\{x_e,y_e\}$ and $x \vee y :=(x_e \vee y _e)_{e\in E(K_n)}$ with $x_e \vee y _e=\max \{x_e,y_e\}$. We will use $d_H$ to denote the Hamming distance  on $\mathcal{G}_n$:  
\begin{equation}\label{hammingdistance}
d_H(x,y):=\sum_{i=1}^M |x_i-y_i|.
\end{equation}

Next, we discuss the crucial notion of the \emph{subgraph counting number}.
For any graph $G$, let $N_G(x)$ be the number of copies of the graph $G$ in the configuration $x$, multiplied by the number of automorphisms of $G$. More precisely:
\begin{equation}\label{homomorphism}
N_G(x)=\sum_{\psi}\prod_{(v,w)\in E(G)}x_{(\psi(v),\psi(w))},
\end{equation}
where the sum is over all injective function $\psi:V(G)\rightarrow [n].$ 
For instance, if $G$ is a triangle, then $N_G(x)$ is the number of \emph{labeled triangles} in $x$.
{For our purposes, we would also need slight variants of the above definition. For a configuration $x$ and an edge $e=(i_1,i_2)$, let $\hat{x}:=x_{e,+}$ and define,
\begin{align} \label{111}
N_G(x,e):=\sum_{\psi}\prod_{(v,w)\in E(G)}\hat{x}_{(\psi(v),\psi(w))},
\end{align} 
where
the sum  is over  all injective function $\psi:V(G)\rightarrow [n]$   satisfying $i_1,i_2\in \psi(V(G))$ and $ (\psi^{-1}(i_1),\psi^{-1}(i_2))\in E(G)$, i.e., $N_G(x,e)$ counts all embeddings of {$G$} into $x_{e,+},$ where some edge of $G$ maps to $e.$
Likewise, for a configuration $x$ and two distinct edges $e=(i_1,i_2)$ and $f=(i_3,i_4)$, letting {$\tilde{x}=x_{ef,+}$,} define,
\begin{align} \label{112}
N_G(x,e,f):=\sum_{\psi}\prod_{(v,w)\in E(G)}\tilde{x}_{(\psi(v),\psi(w))},
\end{align}
where the sum  is again over all injective function $\psi:V(G)\rightarrow [n]$  satisfying $ i_1,\cdots,i_4\in \psi(V(G))$ and $(\psi^{-1}(i_1),\psi^{-1}(i_2)),(\psi^{-1}(i_3),\psi^{-1}(i_4))\in E(G)$. Note that above, $e$ and $f$ are allowed to share a vertex. For $e=f$, set $N_G(x,e,f):=0$. {The latter convention is rather natural because of the following reason.  For the Hamiltonian $H$, defined shortly in \eqref{defhamil}, which will be  considered throughout the paper, for $e\neq f,$ $\partial_{ef}H(x)$ has a natural expression in terms of $ N_{G_i}(x,e,f) $ (see e.g.,\eqref{201}). Since by definition $\partial_{ee}H(x) = 0$, we adopt the convention of defining $N_G(x,e,e):=0$ so that \eqref{201} holds for all $e,f$.}

Letting $\mathcal{P}(\mathcal{G}_n)$ be the collection of probability measures on $\mathcal{G}_n$, for any $\mu\in \mathcal{P}(\mathcal{G}_n)$, $\mu_e(\cdot|x)$ denotes the conditional distribution of the edge $e$ given $\bar{x}_e$. Finally, for any $\mu,\nu \in \mathcal{P}(\mathcal{G}_n)$, let 
 
\begin{equation}\label{totalvariation} 
 d_{TV}(\mu,\nu):=\sup_{A \subset \{0,1\}^M} { |\mu(A)-\nu(A)| }, 
 \end{equation}
denote the total variation distance between $\mu$ and $\nu$, respectively.

\subsection{The Hamiltonian and the Gibbs measure} \label{section 1.2}
We now arrive at the definition of the Gibbs measure on $\mathcal{G}_n$, popularly known as the \it{exponential random graph model} (ERGM), which is the central object of study in this article. Fix $s$ many 
graphs $G_1,\cdots,G_s$ with $|V_i|:=|V(G_i)|$ and $|E_i|:=|E(G_i)|$. Let $a$ be a positive integer satisfying  $|V_i|\leq a$ for $i=1,\cdots,s$.  We define the Hamiltonian $H$  on $\mathcal{G}_n$ by
\begin{align}\label{defhamil}
H(x)=\sum_{i=1}^s \beta_i \frac{N_{G_i}(x)}{n^{|V_i|-2}},
\end{align}
where $\beta_i$ are certain parameters which will be encoded by the vector $\vec{\beta}$.
 Note that the {subgraph} count $N_{G_i}(x)$ is of order $n^{|V_i|}$; thus, the normalization $n^{|V_i|-2}$ ensures that the Hamiltonian is of order $n^2$, {which turns out to be the right scaling for the model.} 
 Finally, the $\mathrm{ERGM}(\vec{\beta})$ (often to be called as just the ERGM) is the Gibbs measure $\pi\in \mathcal{P}(\mathcal{G}_n)$ defined by
\begin{align} \label{erg}
\pi(x) = \frac{1}{Z_n(\vec{\beta})}\exp \Big(\sum_{i=1}^s \beta_i \frac{N_{G_i}(x)}{n^{|V_i|-2}}\Big),
\end{align}
where $Z_n(\vec{\beta})$ is the normalizing constant and the quantity ${f_n(\vec{\beta})}:=\frac{\log Z_n(\vec{\beta})}{{n^2}},$ will be called the \emph{normalized free energy}.
 Throughout the paper, we only consider the case when $\beta_i>0$ for $i= 1,\cdots,s$ so that the Gibbs measure \eqref{erg} is monotone and ferromagnetic. 

\subsection{Glauber dynamics} \label{section 1.3}
 There is a natural (discrete time) Heat-bath Glauber dynamics {(GD)} associated with the ERGM \eqref{erg}, which is defined as follows.
Given the current state $x$, an edge $e$ is  uniformly chosen and resampled according to the conditional distribution of $x_e$ given $\bar x_e$.
It is an easy calculation to verify that given $e$ is chosen to be updated, one has the following transition probabilities: 
\begin{align}\label{tranrates}
P(x,x_{e+})=\frac{\exp(\partial_e H(x))}{1+\exp(\partial_e H(x))} ; P(x,x_{e-})=\frac{1}{1+\exp(\partial_e H(x))},
\end{align}
and as is well known for Glauber dynamics, $\pi$ is reversible with respect to the transition kernel.
For $x\in \mathcal{G}_n$ and edge $e$, let $x^e$ be the configuration $x$ with edge $e$ flipped (i.e. open  $\leftrightarrow$ closed), and define $c(x,e)=P(x,x^e)$  according to \eqref{tranrates}. Then, the
 generator of  Glauber dynamics $\cL$ is defined by
\begin{align}\label{generator}
\mathcal{L} f(x)= \frac{1}{M} \sum_e c(x,e) (f(x^e)-f(x)),
\end{align} 
and the Dirichlet form $\mathcal{E}$  is given by
\begin{align}\label{dirichlet}
\mathcal{E}(f,g):= \frac{1}{M} \sum_{x\in \mathcal{G}_n} \sum_{e} c(x,e) (f(x^e)-f(x)) (g(x^e)-g(x))\pi(x).
\end{align}
Also, for $t\geq 0$, the semigroup  generated by $\mathcal{L}$ will be denoted by:
\begin{align*}
  P_t:=e^{t\mathcal{L}}.
\end{align*}

We next introduce the \it{grand coupling}, which provides a natural coupling between the GD, $(\{Z^x(t)\}_{x\in \mathcal{G}_n})_{t\geq 0},$ starting from all initial configurations $x$ and is often an useful tool to analyze Markov chains.  

{Because of its appearance throughout the paper, it would be convenient to first let 
\begin{equation}\label{impfunction}
\Phi(x) := \frac{e^x}{1+e^x}.
\end{equation}
} Then, for $I$, a uniformly chosen edge, and $U$  a uniform random variable on $[0,1]$ independent of $I$,  for $x\in \cG_n,$ define,
\begin{align}
S^x:=\begin{cases}
1 \quad \text{if} \ 0<U\leq  { \Phi(  \partial_I H(x)  ) }   ,\\
0 \quad \text{if} \   { \Phi(  \partial_I H(x)  ) }      <U\leq 1,
\end{cases}
\end{align}
and proceed to define the next state $Z^x(1),$ as 
\begin{align*}
Z^x(1)=\begin{cases}
x_e \quad \text{if} \  e\neq I, \\
S^x \quad \text{if} \ e=I.
\end{cases}
\end{align*}
The above described one step of the grand coupling can be repeated in a Markovian fashion to obtain the full grand coupling. We will also denote the full and empty configurations by $+$ and $-$ respectively, and $Z^+(t)$ and $Z^-(t)$ to denote the GD started from them.
A particularly useful property of the grand coupling is that it is monotone, i.e., the natural partial order on the configuration space defined earlier is preserved in time under the coupling and hence we will interchangeably refer to this as the \it{monotone coupling}.
Also for any measure $\mu \in \cP(\cG_n),$ we will denote the distribution of the GD at time $t$ starting from $\mu$ by $\mu_t.$ Finally we denote the $\e-$mixing time by 
$$
\tmix(\e):=\inf\{t: d_{TV}(Z^x(t),\pi)\le \e \,\, \forall x\in \cG_n \}.$$

\subsection{High and low temperature phases} To motivate the objects of interest in this paper, distinguish the high and low temperature behaviors and put our results in context, it will be useful to review briefly the prior advances in the study of ERGM. Broadly the two main breakthroughs in this field rely on two different perspectives, large deviation theoretic and the study of GD. Using the former approach, in a seminal work, 
the thermodynamic behavior of ERGM was studied by Chatterjee and Diaconis \cite{cd} who obtained,
a precise variational formula for the free energy of ERGM, $f_n(\vec{\beta})$, via the theory of \it{graphons} which are measurable and symmetric functions $f:[0,1]^2 \rightarrow [0,1]$. 
More formally, generalizing the notion of Gibbs measures on graphs, denoting the space of all graphons modulo composition by measure preserving transformation of $[0,1]$ to itself, by $\widetilde W$,  for a function $T:\widetilde{W} \rightarrow \R$, one can consider the probability distribution $p_n$ on $\mathcal{G}_n$ defined by
\begin{align*}
p_n(G)=\frac{1}{Z_n}e^{n^2T(\tilde{G})},
\end{align*}
where $\tilde G$ is a canonical embedding of $G$ into $\widetilde W,$ and $Z_n$ is the corresponding normalizing constant.
Given the above setting, using a previous result of Chatterjee and Varadhan  \cite{cv} who established a large deviation principle on $\widetilde{W}$, for the sequence of measures  induced by the Erd\H{o}s-R\'enyi $G(n,p)$ graph for fixed $p>0$ (dense case) as $n\to \infty$,
Chatterjee and Diaconis  \cite{cd} showed that,
\begin{align} \label{variational 0}
\lim_{n\rightarrow \infty} \frac{\log Z_n}{{n^2}} = \sup_{\tilde{h}\in \widetilde{W}}\Big(T(\tilde{h})-\frac{1}{2}I(\tilde{h})\Big),
\end{align}
where for any $x\in [0,1],$ we let $$I(x):=-[x\log x+(1-x)\log (1-x)]$$ to be the binary-entropy functional, while for any graphon $W,$ $I(W)$ is simply $\int_{[0,1]^2}I(W(x,y))dxdy $.\\

\noindent
\textbf{Replica-Symmetry:}
In particular, they proved that
the ferromagnetic case $\beta_1,\cdots,\beta_s >0$ falls in the replica symmetric regime i.e., the maximizers  in \eqref{variational 0} are given by the  constant functions, i.e.,
\begin{align} \label{variational1}
\lim_{n\rightarrow \infty}  \frac{1}{n^2}\log Z_n(\vec{\beta}) = \sup_{u\in [0,1]} \Big(\sum_{i=1}^s \beta_i u^{|E_i|} - \frac{1}{2}I(u)\Big).
\end{align}
Informally, the above implies that if  $u_1,\cdots,u_k\in [0,1]$ are the maximizers in \eqref{variational1}, then the ERGM \eqref{erg} behaves like a mixture of the Erd\H{o}s-R\'enyi graphs $G(n,u_i)$'s, in an asymptotic sense. This was made precise in the more recent work by Eldan \cite{e}, and Eldan and Gross \cite{eg}.

To understand the solutions of \eqref{variational1}, recalling the function $\Phi$ from \eqref{impfunction}, for $p>0$, define two functions $\Psi_{\vec{\beta}}$ and $\varphi_{\vec{\beta}}$ by
\begin{align} \label{phi}
\Psi_{\vec{\beta}}(p):=\sum_{i=1}^s 2\beta_i |E_i|  p^{|E_i|-1},\quad \varphi_{\vec{\beta}}(p):= { \Phi(\Psi_{\vec{\beta}}(p)) }=  \frac{\exp(\Psi_{\vec{\beta}}(p))}{1+\exp(\Psi_{\vec{\beta}}(p))}.
\end{align}
{It is not difficult to check that both the above functions are increasing in $p.$ Note that by computing the gradient of the RHS in \eqref{variational1}, all the maximizers of the same, satisfy 
\begin{equation}\label{firstorder}
\varphi_{\vec{\beta}}(p)=p.
\end{equation}
 However, while the solutions of the RHS in \eqref{variational1} represent the constant graphons which account for the dominant contribution to the free energy, there might be several solutions to \eqref{firstorder}, representing more local maxima. It is not difficult to see that the second order condition associated to a $p$ corresponding to a local maxima for \eqref{variational1} is 
 \begin{align}\label{secondorder}
 \varphi_{\vec{\beta}}'(p^*)<1.
 \end{align}  
In fact such states have a natural interpretation in terms of the GD. This perspective was adopted in another major advancement in this field in \cite{bbs} where the authors  pursued an alternate approach of understanding ERGM through studying convergence to equilibrium for the GD.  
Following \cite{bbs}, we say that $\vec{\beta}$ belongs to the  \it{high temperature phase} or is \it{subcritical} if $\varphi_{\vec{\beta}}(p)=p$ has a unique solution $p^*$ and,  which in addition satisfies $\varphi_{\vec{\beta}}'(p^*)<1$. 
Thus in the subcritical phase the mixture of Erd\H{o}s-R\'enyi graphs $G(n,u_i)$ degenerates to a single  Erd\H{o}s-R\'enyi graph $G(n,p^*).$  
Whereas,  $\vec{\beta}$ is said to be in the \it{low temperature phase} or is \it{supercritical}, if $\varphi_{\vec{\beta}}(p)=p$ has at least two solutions $p^*$ with $\varphi_{\vec{\beta}}'(p^*)<1$. 
When $\vec{\beta}$ is neither in the high nor low temperature phase, is called the \it{critical temperature phase}. It was shown in \cite{bbs} that GD is rapidly mixing if $\vec{\beta}$ is subcritical, and is slowly mixing if it is supercritical instead. 
More generally, it is not difficult to, at least heuristically, argue that any local maxima, i.e., $p$ satisfying \eqref{firstorder} and \eqref{secondorder} represents a local metastable state in the following sense: starting from an Erd\H{o}s-R\'enyi graph $G(n,p)$, under the GD evolution, the edge and other subgraph densities stay close to that of the starting state for an exponential in $n$ amount of time. A related version of such a result appears in the low temperature analysis in \cite{bbs}.}

\section{Statements of the results} \label{section 2}
Even though the above results establish in a certain weak sense, that ERGM in the high temperature phase behaves like an  Erd\H{o}s-R\'enyi graph $G(n,p^*)$, several problems remain open, in particular pertaining to how this approximation can be made quantitative.   We list below a few of them:
\begin{enumerate}
\item Does the Glauber dynamics satisfy functional inequalities like Poincar\'e and Log-Sobolev inequalities?
\item What kind of concentration of measure does the ERGM exhibit? 
\item Do natural observables like the number of edges in an ERGM satisfy a central limit theorem? 
\end{enumerate}
Answers to such questions have several potential applications including in testing of hypothesis problems in the statistical study of networks and not surprisingly, this has led to a significant body of work over the recent years. Given the above preparation, we now present our main results and compare them to the progress made in the aforementioned, related works. 
Throughout the sequel, we assume that $\vec{\beta}$ is in the sub-critical (high temperature) phase. {Furthermore, since we will have to deal with different measures, such as the equilibrium measure $\pi,$ the measure induced by the Glauber dynamics, various couplings in the definitions of Wasserstein distances and so on, we will use the same notation $\mathbb{P}$ to denote the underlying measure and the corresponding expectation by $\E$. The exact object will be clear from context and will not lead to any confusion. Finally a random graph of size $n$ will be denoted by $X=(X_e)_{e\in E(K_n)},$ where $X_e$ is the indicator that the edge $e$ is present.} \\  
\noindent
\textbf{Spectral-gap.}
We will denote the spectral gap of the generator $\cL$ defined in \eqref{generator} by $\gamma.$  The following variational characterization is well known: 
\begin{align}\label{variational}
\gamma=\inf_{f}\cE(f,f), \text{ where }f:\cG_n\rightarrow \R \text{ satistifies }\E_{\pi}(f)=0 \text{ and } \E_\pi(f^2)=1. 
\end{align}
It is proved in \cite{bbs}, that the mixing time of GD in the high temperature phase is $\Theta(n^2 \log n)$, which in turn implies that the relaxation time i.e., $\frac{1}{\gamma},$ is $O(n^2 \log n)$.  However it turns out that as a relatively straightforward consequence of already known results, a sharp estimate of the relaxation time throughout the entire high temperature regime follows. This is the first result we record. 
\begin{theorem}\label{Theorem 0} The spectral gap $\gamma$ of $ERGM$ is $\Theta(n^{-2}).$ Thus there exists a constant $C>0$  such that for sufficiently large $n$ and any $f\in L^2(\pi)$,
\begin{align} \label{poincare}
\textup{Var}(f) \leq Cn^2 \mathcal{E}(f,f).
\end{align}
\end{theorem}
{By standard theory of reversible Markov chains {\cite[Chapter 12]{lpw}}, the above  implies that the variance of  $P_tf$ decays exponentially fast along the GD, i.e., there exists $c>0$ such that for sufficiently large $n$ and any $t>0$,
\begin{align*}
\text{Var}(P_tf) \leq \Big(1-\frac{c}{n^2}\Big)^t \text{Var}(f).
\end{align*}
A similar bound holds for the Dirichlet form $\mathcal{E}(P_tf,P_tf)$ as well,
\begin{align*}
\mathcal{E}(P_tf,P_tf) \leq \Big(1-\frac{c}{n^2}\Big)^t \mathcal{E}(f,f).
\end{align*}
}
\noindent
\textbf{Concentration of measure.} Concentration of measure properties of ERGM are the focus of this article.  Often in spin systems, for very high temperature values, the model falls in  DU regime (see Section \ref{section 3.1} for definition), where certain classical perturbation arguments can be used to prove Gaussian concentration for Lipschitz functions. While for the ERGM, similar arguments were presented in \cite{ss},  going beyond the DU regime has resisted rigorous mathematical analysis so far.  Before making precise statements, we first define the notion of a Lipschitz function on $\mathcal{G}_n$: for an $M$-dimensional vector $\textbf{v}=(v_1,\cdots,v_M)$ with $v_i\geq 0$, and $f: \cG_n \mapsto \R$ we say that $f \in \text{Lip}(\textbf{v})$ if 
\begin{align}\label{lipdef}
|f(x)-f(y)| \leq \sum_{i=1}^M v_i \1\{x_i\neq y_i\}.
\end{align} 

A metric measure space $(X,d,\mu)$ is said to admit an \it{exponential concentration}  if  for some $u>0$,
\begin{align*}
\sup_{\norm{f}_{\text{Lip}}\leq 1, \ \int f d\mu = 0}\int_X e^{uf} d\mu < \infty.
\end{align*}
 It is a classical fact that when $X$ is a Riemannian manifold with metric $d$ and  induced volume measure $\mu$, the existence of a positive spectral gap for the Laplacian  implies the exponential concentration of $(X,d,\mu)$. We refer to the monographs \cite{blm,l} for more details.

The discrete analog of the above fact 
 was obtained by Ledoux \cite{l2}. More formally, let $X$ be a finite (or countable) set, and $P:X\times X\rightarrow \R$ be a Markov kernel. Also assume that a probability measure $\mu$  is reversible with respect to $P$. Then, we define the following notion of the Lipschitz constant of a function $f:X\rightarrow \R$:
 \begin{align} \label{lip}
 \norm{f}_{\text{Lip}}^2 := \sup_{x\in X} \sum_{y\in X} |f(x)-f(y)|^2 P(x,y),
\end{align} 
and  then define the canonical distance on $X$ by
\begin{align*}
d(x,y) := \sup_{\norm{f}_{\text{Lip}}\leq 1, \ \int f d\mu = 0} [f(x)-f(y)].
\end{align*}
It is proved in \cite[Theorem 2.2]{l2} that if  $\lambda_1$ is a spectral gap of the Markov chain, then 
\begin{align*}
\sup_{\norm{f}_{\text{Lip}}\leq 1}\int_X e^{\sqrt{\lambda_1/2} f} d\mu < 3.
\end{align*}
In particular, $(X,d,\mu)$ admits an exponential concentration if $\lambda_1>0$.
Applying this fact in conjunction with the already established Poincar\'e inequality \eqref{poincare} for the ERGM, allows one to conclude that there exists a constant $C>0$ such that  for any $f$ with  $\norm{f}_{\text{Lip}} \leq 1$ (with respect to the Lipschitz norm \eqref{lip}),
\begin{align*}
\mathbb{P}(|f(X)-\mathbb{E} f(X)|>t) \leq 3e^{-Ct / n}.
\end{align*}
The next theorem is the main result of this article which improves the above to establish the expected Gaussian concentration result throughout the high temperature phase.

\begin{maintheorem}  \label{Theorem 2}  There exists a constant $c>0$  such that for sufficiently large $n$, for any $f\in \textup{Lip}(\textup{\textbf{v}})$ and $t\geq 0$,
{\begin{align}
\mathbb{P}(|f(X)-\mathbb{E} f(X)|>t) \leq 2e^{-ct^2 / \norm{\textup{\textbf{v}}}_1\norm{\textup{\textbf{v}}}_\infty} .
\end{align}}
\end{maintheorem}
This is the first result of its kind, and we expect that the method could be useful in  other related settings.  We next move on to a key application of the above.\\

\noindent
\textbf{Central limit theorem.} 
Going beyond concentration of measure, establishing 
central limit theorems (CLT) has been a fundamental problem in the study of classical Gibbsian systems. For instance, CLT  for the magnetization has been obtained for the subcritical (finite range) Ising model, using exponential decay of correlations (see \cite{e1,gw,n,n1} for more details). A similar result holds for the exactly solvable mean-field Ising model (Curie-Weiss) for all sub-critical $\beta<1$ (see \cite{em} for more details). 

However, unlike classical spin systems, as has been alluded to before, there are several barriers to proving a CLT type result in the ERGM setting.  {The only result we are aware of is for the ERGM with the Hamiltonian being the two-star counting function (two-star model) where \cite{mukherjee1} established a full CLT heavily exploiting a special Ising-type interaction the model admits. }

Our next result makes partial progress in this front.  In particular, we prove a CLT \textit{throughout the high temperature regime} for the number of open bits restricted to a sub-linear (in $n$) number of coordinates. To the best of the authors' knowledge, this is the first CLT of any kind in this {general} setting. While the full CLT for the number of edges is still open,  refined versions of the methods introduced in this paper might be useful in resolutions of the same and other related questions. 

\begin{maintheorem} \label{Theorem 3} For any  sequence  of positive integers $m$ satisfying $m=o(n)$ and $m\rightarrow \infty$ as $n\rightarrow \infty$ we have the following. Consider any set of 
$m$ different edges $i_1,\cdots,i_m \in E(K_n)$ that do not share a vertex. Then, the following central limit theorem for the normalized number of open edges among $i_1,\cdots,i_m$ holds:
\begin{align}
\frac{X_{i_1} + \cdots+X_{i_m}  - \mathbb{E} [X_{i_1} + \cdots+X_{i_m} ]}{\sqrt{\textup{Var}(X_{i_1} + \cdots+X_{i_m})}} \xrightarrow{d} N(0,1),
\end{align}
where $\xrightarrow{d}$ denotes weak convergence.
\end{maintheorem}

\noindent
\textbf{Bounding the $W_1-$Wasserstein distance. } Our final result is another consequence of the Gaussian concentration result Theorem \ref{Theorem 2}. Namely, we provide an alternate proof of a quantitative estimate on the closeness between a high temperature ERGM and an Erd\H{o}s-R\'enyi graph obtained in \cite{rr}. 
It is not hard to see from \eqref{phi}, that in the multiparameter case $s\geq 2$, the set of $\vec\beta = (\beta_1,\cdots,\beta_s)$'s, in the high temperature phase, sharing the same $p^*$, form a $(s-1)$ dimensional hyperplane solving
\begin{align} \label{901}
p^* = \frac{\exp(\sum_{i=1}^s 2\beta_i |E_i| (p^*)^{|E_i|-1})}{1+\exp(\sum_{i=1}^s 2\beta_i |E_i| (p^*)^{|E_i|-1})}.
\end{align}
Thus from the discussion so far in this article, for all these $\vec \beta$, the ERGM behaves qualitatively like the  Erd\H{o}s-R\'enyi graph $G(n,p^*)$ in various senses. A particular way to quantify the same  is through the notion of contiguity\footnote{Given two sequences of measures $\{p_n\}, \{q_n\}$ where $p_n$ and $q_n$ are defined on the same space, one says that the former is contiguous with respect to the latter, if for any sequence of events $A_n$, $q_n(A_n)\to 0$ implies $p_n(A_n)\to 0.$}, leading one to speculate whether any  ERGM($\vec\beta$) satisfying \eqref{901}, is contiguous with the Erd\H{o}s-R\'enyi graph $G(n,p^*)$. {While such results have been indeed recently  proved for the multi-parameter Ising model in \cite{gm}, for the ERGM, in the special case where the Hamiltonian $H(x)$ depends only on the number of edges and the two-star count, \cite{mukherjee1} in fact constructed a consistent estimator of $\vec{\beta}$ for \emph{all} values of the latter, proving that the two measures are in fact not contiguous. While a general result refuting contiguity is still missing, currently in the other direction, in the sub-critical case, we present a quantitative bound  in terms of the $W_1-$Wasserstein distance between  ERGM($\vec\beta$) and   $G(n,p^*)$, where for 
$\mu,\nu$ in $\cP(\cG_{n})$: 
\begin{equation*}
W_1(\mu,\nu)=\inf_{\Gamma}\sum_{e\in E(K_n)} \P_{\Gamma}(X_e\ne Y_e),
\end{equation*}
where the infimum is obtained over all couplings {$\Gamma$} of $\mu$ and $\nu.$

\begin{maintheorem}\label{w1distance}
There exists a constant $C>0,$ such that for sufficiently large $n$,
 \begin{align} \label{103}
W_1(\pi,\nu) \leq C  n^{3/2} \sqrt{\log n},
\end{align} 
where $\pi$ is the ERGM($\vec\beta$) distribution, and $\nu\in \mathcal{P}(\mathcal{G}_n)$ is the Erd\H{o}s-R\'enyi distribution $G(n,p^*)$.
\end{maintheorem}
The above result is of relevance in the study of estimability properties of the ERGM and a related discussion is presented later in Section \ref{pseudo-section}. 
\begin{remark} \label{remark 7.4} 
A slightly stronger statement than \eqref{103}, without the logarithmic term $\sqrt{\log n}$, is obtained in \cite{rr}. {It is shown that for any sub-critical ERGM$(\vec\beta)$,} there exists a constant $C>0$ such that for sufficiently large $n$ and any Lipschitz function $f$,
\begin{align} \label{731}
|\mathbb{E}f(X) - \mathbb{E}f(Y)| \leq C \norm{f}_{\text{Lip}} \cdot n^{3/2},
\end{align}
where, the Lipschitz constant $\norm{f}_{\text{Lip}}$ is defined as an infimum of values $c>0$ such that
\begin{align*}
|f(x)-f(y)| \leq c \sum_{i=1}^M |x_i-y_i|
\end{align*}
for any $x,y\in \mathcal{G}_n$.
The key step in \cite{rr} to obtain \eqref{731} was to use the \it{Stein's method for CLT}  and obtain a solution $h$ to Stein equation 
\begin{align*}
\mathcal{L}h(x) = f(x) - \mathbb{E}f(X)
\end{align*}
(recall that $\mathcal{L}$ is a generator associated with ERGM(${\vec{\beta}}$)). 
\end{remark}

\section{Main ideas of the proofs and organization of the article}\label{iop}
We discuss the key ideas in the proofs of the theorems stated above. 
\subsection{Poincar\'e inequality}The proof of this, in fact, follows in a relatively straightforward fashion from known results. Nonetheless, we include it for completeness.
To understand the spectral gap, we rely on its relation to temporal mixing properties of the GD, which was the subject of study in \cite{bbs}. It is a classical fact (see e.g., \cite{lpw}) that
 a Markov chain exhibiting a strict contraction property (described below), has a positive spectral gap. In particular, 
this is verifiable if the temperature $\vec\beta$ lies in DU regime. In fact, if the $\beta_i'$s are small enough so that $\Psi'_\beta(1)<4$ (see \eqref{phi}), then there exists a constant $\delta>0,$ such that for all $x,y \in \cG_n$ some coupling of {$Z^x(1)$ and $Z^y(1)$},
\begin{align} \label{213}
\mathbb{E} d_H(Z^x(1),Z^y(1)) < \Big(1-\frac{\delta}{n^2}\Big)d_H(x,y)
\end{align}
(more explanations are provided in Section \ref{section 3.1}).
This immediately implies that the relaxation time is $O(n^2)$ in  DU regime.

However, \eqref{213} does not hold throughout the high temperature phase. Nonetheless, it is proved in \cite{bbs} that there exists a constant $c>0$ such that something quite similar does hold. Namely, {under the monotone coupling $(Z^+(t),Z^-(t))$,} for each $t\geq cn^2$,
\begin{align} \label{214}
\mathbb{E}[ d_H(Z^+(t+1),Z^-(t+1))] < \Big(1-\frac{\delta}{n^2}\Big) {\mathbb{E}}[d_H(Z^+(t),Z^-(t))]+e^{-\Omega (n)}.
\end{align}
This allows us to control $d_{TV}(Z^+(t),Z^-(t))$ at  $t=n^3$. Then,  the sub-multiplicative property of $d_{TV}(\cdot,\cdot),$ implies that there exists a constant $\alpha>0$ such that for each $t\geq n^3$,
\begin{align*}
d_{TV}(Z^+(t),Z^-(t)) \leq e^{-\alpha t/n^2}.
\end{align*}
This and standard relations between total variation distance and the spectral gap then allows one to quickly deduce that the relaxation time is $O(n^2).$ The lower bound follows by proving an upper bound on the spectral gap using the variational formula \eqref{variational} and a suitable test function.

\subsection{Gaussian concentration} Theorem \ref{Theorem 2} is the main result of this paper and the key tool we rely on, is a version of \it{Stein's method for  concentration} via construction of suitable exchangeable pairs, developed in \cite{c,c1} (see Section \ref{section 7.1} for details). This combined with temporal mixing results derived from \cite{bbs} allows us to obtain the sought bounds. Informally, to make this strategy work, one has to control the $L^\infty$-norm of the  function
\begin{align} \label{231}
F(x,y):=\sum_{t=0}^\infty (P_tf(x)-P_tf(y)).
\end{align} 

In particular, for our application, it will suffice to consider $y$ which differs from $x$ only at one edge, say $e$.  
Thus, to control the RHS in \eqref{231}, using the Lipschitz-ness of $f$, with $x,y$ as above, it suffices to bound the $L_1$ norm of the vector
\begin{equation}\label{hammingvector}
 \textbf{r}(t) = (\mathbb{P}(Z^x(t)_e\neq Z^y(t)_e))_{e\in E(K_n)},
 \end{equation} for the monotone coupling $(Z^x(t), Z^y(t))_{t\geq 0}$. More precisely, we will show the following bounds for $\|r(t)\|_1$ which will allow us to bound the sum on the RHS of \eqref{231}: {for some $C,d,\delta,\alpha>0$,}
\begin{align*}
\|r(t)\|_1\le \left\{\begin{array}{cc}C& t\le dn^2,\\
(1-\frac{\delta}{n^2})^{t-dn^2} & dn^2\le t\le n^3,\\
e^{-\alpha t / n^2} & t\ge n^3.
\end{array}
\right.
\end{align*}

The first one is a crude bound derived from straightforward considerations, the second bound follows from a contraction result deduced from the estimates in \cite{bbs}, while the third bound follows from the second by sub-multiplicative properties of total variation distance.

\subsection{Central limit theorem} The strategy to prove the central limit theorem is to estimate the correlation structure of the edge variables $X_e.$  This is a rather delicate task and in particular, we can obtain useful estimates only when the edges are vertex disjoint. More precisely, we prove the following  key $k$\it{-correlation}  estimate: 
if $k$ many distinct edges $i_1,\cdots,i_k$ do not share a vertex, then
\begin{align}\label{kcorr1}
\big\vert \mathbb{E} [(X_{e_{i_1}}-\mathbb{E} X_{e_{i_1}})\cdots (X_{e_{i_k}}-\mathbb{E}X_{e_{i_k}})] \big\vert= O\Big(\frac{1}{n^{k/2}}\Big).
\end{align}

In order to establish the above bound, we first obtain quantitative estimates on the total variation distance between the distribution of an edge variable conditioned on other edges and the unconditional marginal (Proposition \ref{prop 5.1}) as applications of the Gaussian concentration result and FKG inequality. In fact, we show that the fluctuation arising from conditioning on a fixed number of other edges is at most $O(\frac{1}{n})$:
\begin{align}\label{quantind}
\Big\vert\mathbb{P}(X_{e_1} = 1) -  \mathbb{P}(X_{e_1} = 1 | X_{e_2}=a_1,\cdots,X_{e_k}=a_k) \Big\vert= O\Big(\frac{1}{n}\Big).
\end{align}
The $k-$correlation estimate is then proved using the above result and a conditional-version of Theorem \ref{Theorem 2}. This allows us to complete the proof using a moment method argument. 
\subsection{Bounding $W_1-$Wasserstein distance:} The proof proceeds by constructing a coupling of the stationary GD on the ERGM and $G(n,p^*)$ respectively,  where the same edge is updated in the two Markov chains and the update probabilities are coupled in an optimal way. We then show as an application of the Gaussian concentration result, that with high probability the total variation distance between the update probabilities of the updated edge in the two models is no more than 
$O(\sqrt{\frac{\log n}{n}}),$ throughout the time it takes to update all the edges. The above then implies that the amount of discrepancy induced between the two chains till all the edges are updated is no more than $O(\sqrt{\log n}n^{3/2})$, which along with stationarity of the chains is enough to finish the proof.

\subsection{Organization of the paper}
Several general facts about ERGM in the high temperature phase are reviewed in Section \ref{section 3}; while Dobrushin's uniqueness regime is highlighted in Section \ref{section 3.1}, the general high temperature regime is discussed in Section \ref{section 3.2}. The short proof of the Poincar\'e inequality appears towards the end of this section. 

The proof of main result in this article establishing Gaussian concentration for Lipschitz functions appears in Section \ref{section 7}. The subsequent Sections \ref{section 8} and \ref{section 9} contain the proofs of the central limit theorem (Theorem \ref{Theorem 3}) and the bound on the Wasserstein distance to Erd\H{o}s-R\'enyi graphs (Theorem \ref{w1distance}) respectively. 

\subsection{Acknowledgement} The authors thank Charles Radin and Arthur Sinulis for useful comments. They also thank two anonymous referees for their detailed and thoughtful comments that helped improve the paper.
SG's research is partially supported by a Sloan Research Fellowship in Mathematics and NSF Award DMS-1855688. KN's research is supported by a summer grant of the
UC Berkeley Mathematics department.

\section{ERGM in the high temperature regime} \label{section 3}
To motivate some of our arguments and set up further necessary notation, we review some useful facts about the ERGM in the high temperature regime. We begin by considering the  perturbative case \cite{d1,d2}.
\subsection{ Dobrushin's uniqueness regime} \label{section 3.1}
For edges $e$ and $f$, let 
\begin{align} \label{interdepend}
a_{ef}:=\sup_{x\in \mathcal{G}_n} d_{TV}(\pi_e(\cdot|x_{f+}),\pi_e(\cdot |x_{f-})) .
\end{align}
The \it{Dobrushin interdependent matrix} $A$, of size $M\times M$, is defined as  $A=(a_{ef})_{e,f \in E(K_n)}$. We say that the Gibbs measure $\pi$ satisfies the $L^2$-version of \it{Dobrushin's uniqueness condition} if the matrix $A$ satisfies $\norm{A}_{2} <1$. This condition is slightly different from the original $L^1$-version of  Dobrushin's uniquness condition, where the matrix $A$ is assumed to satisfy $\norm{A}_1<1$, ($\norm{A}_{2}$ and $\norm{A}_{1}$ denote the norms of $A$, thought of as an operator from $\R^{M}$ to itself equipped with the $\ell_2$ and $\ell_1$ norms respectively). {
Recalling \eqref{tranrates} and the fact that the total variation distance between two probability measures $\mu$ and $\nu$ on the  two spin system $\{0,1\}$ is $|\mu(\{1\}) - \nu(\{1\})|$,    for any configuration $x$,}
\begin{align}  \label{201}
 d_{TV}&(\pi_e(\cdot|x_{f+}),\pi_e(\cdot |x_{f-}))=|\pi_e( x_e=1|x_{f+}) - \pi_e(x_e=1|x_{f-})| \nonumber \\
&=  \frac{1}{2}\big(1+\tanh (\partial_e H(x_{f+})/2)\big) - \frac{1}{2}\big(1+\tanh (\partial_e H(x_{f-})/2)\big) \nonumber \\
&\leq \frac{1}{4}|\partial_{ef}H(x)| {=} \frac{1}{4} \sum_{i=1}^s \beta_i \frac{N_{G_i}(x,e,f)}{n^{|V_i|-2}},
\end{align} 
where the above notations were introduced around \eqref{homomorphism}. Here, the last inequality above follows from the mean value theorem, the definition of $\partial_{ef}H(x)$ and that $\sup_{x\ge 0} \frac{{\rm{d}}}{{\rm d}x}\tanh(x)=1$. {The last equality in \eqref{201} follows from the following expressions of the discrete derivatives. Recalling the definitions from \eqref{discretederivatives1}, we have
\begin{align} \label{802}
 \partial_e H(x) =  \sum_{i=1}^s \beta_i \frac{  N_{G_i}(x,e)}{n^{|V_i|-2}}, 
\end{align}
and hence,
\begin{align*}
\partial_{ef}H(x) = \partial_f (\partial_e H(x)) =  \sum_{i=1}^s \beta_i \frac{ \partial_f N_{G_i}(x,e)}{n^{|V_i|-2}} = \sum_{i=1}^s \beta_i \frac{N_{G_i}(x,e,f)}{n^{|V_i|-2}}.
\end{align*}
 } 
We now define an $M\times M$ symmetric matrix $L = (L_{ef})$ by
\begin{align}\label{crucialdef23}
L_{ef}:=\frac{1}{4} \sum_{i=1}^s \beta_i \frac{ N_{G_i}(K_n,e,f)}{n^{|V_i|-2}}.
\end{align}
Then, by \eqref{201}, we have that for any edges $e$ and $f$,
\begin{align} \label{202}
a_{ef}&\leq L_{ef}.
\end{align}
Using the fact that 
\begin{equation}\label{keyineq100}
\displaystyle{\sum_{e: e\neq f} N_G(K_n,e,f)} = (|E|-1)N_G(K_n,f)=2|E|(|E|-1) \binom {n-2} {|V|-2} (|V|-2)!,
\end{equation}
we conclude that 
\begin{align*}
\sum_{e: e\neq f} L_{ef}  = \frac{1}{4} \sum_{i=1}^s  \beta_i \frac{ \sum_{e:e\neq f}  N_{G_i}(K_n,e,f)}{n^{|V_i|-2}} < \frac{1}{2}\sum_{i=1}^s \beta_i |E_i|(|E_i|-1).
\end{align*}
 Since $L$ is symmetric, this implies that 
\begin{align} \label{211}
 \sup_n  \norm{L}_2 \leq  \sup_n  \norm{L}_1 < \infty.
\end{align}
  Assume that $\vec\beta=(\beta_1,\ldots,\beta_s),$ satisfies the condition
\begin{align} \label{dobrushin}
\frac{1}{2}\sum_{i=1}^s \beta_i |E_i|(|E_i|-1)<1.
\end{align}
Then, by \eqref{202},
\begin{align*}
\norm{A}_2\leq  \norm{L}_2\leq \norm{L}_1 <1,\quad \norm{A}_1\leq \norm{L}_1<1.
\end{align*}

Thus, all $\vec\beta$ satisfying \eqref{dobrushin} lies in the ($L^2$-version of) DU regime.
Although in such regimes, for classical spin systems on lattices, it is well known the stationary measure behaves qualitatively like a product measure in a rather strong sense, satisfying the LSI and other related concentration of measure properties (see e.g. \cite{cesi,k,sz0,sz,z,m} for more details), for the ERGM, as mentioned earlier such an analysis was carried out only very recently in \cite{ss}.

On the other hand,  the $L^1$-version of DU condition,  $\norm{A}_1<1$, implies that there exists a coupling for the GD such that the Hamming distance strictly contracts. {Although this is a well-known fact, we include the brief proof for the sake of completeness. 
Suppose that two configurations $x,y \in \cG_n$ differ at edges $j_1,\cdots,j_l$, and take a path of configurations $x=w^0,w^1,\cdots,w^{l-1},w^l=y$ such that $w^{i-1}$ and $w^{i}$ differ  only  at the edge $j_i$. Then,  under the grand coupling $(Z^x(t),Z^y(t))$, 
\begin{align} \label{212}
\mathbb{P}(Z^x_e(1)\neq Z^y_e(1)) &= \Big(1-\frac{1}{M}\Big)\1\{x_e\neq y_e\} + \frac{1}{M} d_{TV} (\pi_e(\cdot|x),\pi_e(\cdot |y)) \nonumber \\
& \leq \Big(1-\frac{1}{M}\Big)\1\{x_e\neq y_e\} + \frac{1}{M} \sum_{i=1}^l d_{TV} (\pi_e(\cdot|w^{i-1}),\pi_e(\cdot |w^i)) \nonumber \\
& \leq \Big(1-\frac{1}{M}\Big)\1\{x_e\neq y_e\} + \frac{1}{M} \sum_{i=1}^l  a_{e j_i}   \nonumber  \\
&= \Big(1-\frac{1}{M}\Big)\1\{x_e\neq y_e\} + \frac{1}{M} \sum_{f:f\neq e} a_{ef}  \1\{x_f\neq y_f\}.
\end{align} }
Adding these inequalities over all edges $e$, we obtain 
\begin{align} \label{contraction}
\E [d_H(Z^x(1), Z^y(1))] \leq \Big(1-\frac{1-\norm{A}_1}{M}\Big) d_H(x,y),
\end{align}
proving a strict contraction in the Hamming distance if  $\norm{A}_1<1$. This also implies bounds on the spectral gap (see  the monograph \cite{lpw} for details).

{Note that in general, by  \eqref{202} and \eqref{contraction}, for any two configurations $x,y \in \cG_n$, under the grand coupling  $(Z^x(t), Z^y(t))$,
\begin{align*}
\E [d_H(Z^x(1), Z^y(1))] \leq \Big(1-\frac{1-\norm{L}_1}{M}\Big) d_H(x,y),
\end{align*}
which by induction implies 
 \begin{align} \label{contraction2}
\E [d_H(Z^x(t), Z^y(t))] \leq \Big(1-\frac{1-\norm{L}_1}{M}\Big)^t  d_H(x,y).
\end{align}
 }
 \begin{remark}\label{LSIdiscussion}
It is proved in \cite[Theorem 6.2]{cd} that when $\vec{\beta}$ satisfies \eqref{dobrushin},  the model is in the replica symmetric phase even for negative values of $\beta_i$ i.e.,  the maximizers in \eqref{variational 0} are constant functions provided that
$
\frac{1}{2}\sum_{i=1}^s |\beta_i| |E_i|(|E_i|-1)<1.
$ {In fact, under the same condition which falls in the DU regime, as mentioned earlier, \cite{ss} established the LSI and derived concentration of measure properties.}
 \end{remark}

\subsection{Sub-critical but beyond Dobrushin's uniqueness regime} \label{section 3.2}
Even beyond DU regime, where strict contraction property of GD no longer holds, it was nonetheless proved in \cite{bbs} that there exists a set $\mathcal{T}\in \cG_n,$ such that 
\begin{enumerate}
\item GD starting from two states in $\mathcal{T}$ exhibits a strict contraction. 
\item From any starting state, GD hits $\cT$ with high probability within $O(n^2)$ steps.  
\end{enumerate}
We will need precise versions of the above statements to prove Theorem \ref{Theorem 0}. To this end, it would be convenient to introduce the notion of normalized subgraph counting number, following  \cite{bbs}. For any configuration $x$, edge $e$, and a graph $G$, define  $r_G(x,e)$ by
\begin{align} \label{320}
r_G(x,e):=\Big(\frac{N_G(x,e)}{2|E|n^{|V|-2}}\Big)^{1/(|E|-1)},
\end{align}
and then define
\begin{align}\label{defga}
r_{1,\text{max}}(x):=\max_{e,G\in \mathbb{G}_a}r_G(x,e),\quad r_{1,\text{min}}(x):=\min_{e,G\in \mathbb{G}_a}r_G(x,e),
\end{align}
where $\mathbb{G}_a$ denotes a collection of all graphs with at most $a$ many vertices. Recall from \eqref{defhamil} that $a$ is the number satisfying $|V_i|\leq a$ for $i=1,2,\cdots,s$. 
Recalling $p^{*}$ from \eqref{variational}, and ${\delta}>0$, let $\mathcal{T}_{{\delta}}$  be the collection of configurations defined by
\begin{align}\label{t} 
\mathcal{T}_{{\delta}}:=\{x \in \cG_n: p^*-{\delta}<  r_{1,\text{min}}(x)\leq r_{1,\text{max}}(x)<  p^*+{\delta}\}.
\end{align}
The following is one of the key results in \cite{bbs}. 
\begin{theorem}\cite[{Lemma 16}]{bbs}  \label{theorem 2.1}
Suppose that $\vec{\beta}$ lies in the high temperature regime. Then, for any ${\delta}>0$, there exists a constant $c$  such that for any initial configuration $x\in \mathcal{G}_n$, when $t\geq cn^2$,
\begin{align*}  
\mathbb{P}(Z^x(t)\in \cT_{ {\delta}} ) \geq 1-e^{-\Omega (n)}.
\end{align*}
\end{theorem}

Now, we introduce a generalized version of the Dobrushin's matrix defined in Section \ref{section 3.1}, where instead of taking  a supremum over all configurations in \eqref{interdepend}, we keep the dependence on the configuration $x$ to define $A(x):=(a_{ef}(x))_{e,f\in E(K_n)}$ by,
\begin{align} \label{interdepend2}
a_{ef}(x):=d_{TV}(\pi_e(\cdot|x_{f+}),\pi_e(\cdot |x_{f-})).
\end{align}
Then,
by definition,
\begin{align} \label{aef}
a_{ef}(x) = \Phi(\partial_e H(x_{f+} ) ) -  \Phi(\partial_e H(x_{f-})).
\end{align}

For two configurations $y\leq z$, consider the monotone coupling of $Z^y(t)$ and $Z^z(t)$, and define the $M$-dimensional vector $\textbf{r}(t) =  (\textbf{r}(t)_e)_{e\in E(K_n)} $ by
\begin{align*}
\textbf{r}(t)_e  = \mathbb{P}(Z^y_e(t) \neq Z^z_e(t)).
\end{align*}

 The proof of \cite[Lemma 18]{bbs} implies the following proposition, which states the pointwise decay of the vector $\textbf{r}(t)$ along the GD starting from the set $\mathcal{T}_{{\delta}}$.
 
{ 
\begin{proposition} \label{prop 4.3}
Suppose that 
$\vec{\beta}$ lies in the high temperature phase. Then,
there exists a  constant $\delta>0$ such that  the following holds for sufficiently large $n$.
For any $y\leq z$ with $y,z\in \mathcal{T}_\delta$,  there is a $M\times M$ matrix $K  = K(n,y,z)$ with
\begin{align*}
\norm{K}_1 < 1-\frac{\delta}{M}
\end{align*}
and
\begin{align} \label{431}
\textup{\textbf{r}}(1)\leq  K \textup{\textbf{r}}(0),
\end{align}
where the inequality is pointwise.
\end{proposition}
}

Since neither the statement or its proof explicitly appears in \cite{bbs}, we include the details for completeness.
\begin{proof}

{The proof follows from the following result:  if $\vec{\beta}$ lies in the high temperature phase, then there exists sufficiently small $\delta>0$ such that for sufficiently large $n$:
\begin{align} \label{205}
x\in \mathcal{T}_\delta \Rightarrow \norm{A(x)}_1 < 1-\delta.
\end{align}
To see the above, 
by mean value theorem, for some {$\partial_e H(x_{f-})\leq u_0\leq \partial_e H(x_{f+})$}, 
\begin{align} \label{221}
a_{ef}(x) &= \Phi(\partial_e H(x_{f+}) ) -  \Phi(\partial_e H(x_{f-}))   =  \Phi'(u_0)  (\partial_e H(x_{f+})-\partial_e H(x_{f-})) .
\end{align}
Note that  for any edges $e$ and $f$,
\begin{align} \label{226}
r_G (x_{f+},e) - r_G (x_{},e)  = o(1), \quad r_G (x_{},e) - r_G (x_{f-},e)  = o(1).
\end{align}{
In fact, the quantity  $N_G (x_{f+},e) - N_G (x_{},e)$ is $  O(n^{|V(G)|-3})$ if $e$ and $f$ share an edge and $O(n^{|V(G)|-4})$ if $e$ and $f$ do not share an edge (similar property holds for $N_G (x_{},e) - N_G (x_{f-},e)$ as well). Using the inequality $(z_1+z_2)^p - z_1^p \leq z_2^p$ for $z_1,z_2\geq 0, \ 0<p\leq 1$, and recalling the normalization $n^{|V(G)|-2}$ in \eqref{320}, we have \eqref{226}. Hence,}
since $x\in \mathcal{T}_\delta$,  for sufficiently large $n$, we have $x_{f+},x_{f-} \in \mathcal{T}_{2\delta}$ for any edge $f$.
 
   Thus recalling \eqref{phi}, {and using
 \begin{equation}\label{knineq}
N_G(K_n,g)=2|E| \binom {n-2} {|V|-2} (|V|-2)!= (1-o(1))2|E| n^{|V_i|-2}\leq 2|E| n^{|V_i|-2},
\end{equation}} we obtain, 
\begin{align*}
\partial_e H(x_{f+})  \leq  \sum_{i=1}^s \frac{\beta_i (p^*+2\delta)^{|E_i|-1} N_{G_i}(K_n,e)}{n^{|V_i|-2}} \overset{\eqref{knineq}}{\leq} \Psi_{\vec{\beta}} (p^*+2\delta),
\end{align*}
and
\begin{align*}
\partial_e H(x_{f+}) \geq  \sum_{i=1}^s \frac{\beta_i (p^*-2\delta)^{|E_i|-1} N_{G_i}(K_n,e)}{n^{|V_i|-2}} = (1-o(1)) \Psi_{\vec{\beta}} (p^*-2\delta).
\end{align*}
Since the same inequalities also hold for $\partial_e H(x_{f-})$,  using \eqref{221} and the fact that $ \Phi'$ is decreasing on $(0,\infty)$, for any $\eta>0$, for sufficiently large $n$ and small $\delta>0$,
\begin{align} \label{222}
a_{ef}(x) &\leq (1+\eta)  \Phi'(\Psi_{\vec{\beta}} (p^*))  (\partial_e H(x_{f+})-\partial_e H( x_{f-} )) \nonumber \\
&=(1+\eta)  \Phi'(\Psi_{\vec{\beta}} (p^*))      \sum_{i=1}^s \beta_i \frac{N_{G_i}(x,e,f)}{n^{|V_i|-2}} .
\end{align}
 
Since $\varphi_{\vec{\beta}}'(p^*)<1$ (see \eqref{secondorder}), one can
take constants ${\eta},\delta>0$ sufficiently small  satisfying  
\begin{align} \label{323}
(1+{\eta})^2 \varphi_{\vec{\beta}}'(p^*)<1-\delta.
\end{align}
For such ${\eta}$, define a $M\times M$ symmetric matrix $U(x)=(U_{ef}(x))$  by
\begin{align} \label{325}
U_{ef} (x):= (1+{\eta})  \Phi'(\Psi_{\vec{\beta}} (p^*))  \sum_{i=1}^s \beta_i \frac{N_{G_i}(x,e,f)}{n^{|V_i|-2}} .
\end{align}
Then, by \eqref{222}, for sufficiently  large $n$ and sufficiently small $\delta>0$,
\begin{align} \label{324}
a_{ef}(x) \leq U_{ef}(x).
\end{align}
We verify that the condition \eqref{323} ensures that
\begin{align} \label{326}
\norm{U(x)}_1<1-\delta.
\end{align}  In fact,  denoting $G_g$ by a graph obtained from $G$ by removing an edge $g$,  since $x\in \mathcal{T}_\delta$,  using \cite[Lemma 10]{bbs} and \eqref{keyineq100}, for sufficiently small $\delta>0$,

\begin{align*}
 \sum_{e:e\neq f} \sum_{i=1}^s \beta_i \frac{N_{G_i}(x,e,f)}{n^{|V_i|-2}} &=  \sum_{i=1}^s  \frac{\beta_i \sum_{g\in E(G_i), g\neq f} N_{(G_i)_g}(x,f)  }{n^{|V_i|-2}} \\
 &\leq     \sum_{i=1}^s  \frac{\beta_i \sum_{g\in E(G_i), g\neq f} (p^*+\delta)^{|E_i|-2} N_{(G_i)_g}(K_n,f)  }{n^{|V_i|-2}} \\
 &=  \sum_{e:e\neq f} \sum_{i=1}^s \beta_i (p^*+\delta)^{|E_i|-2}  \frac{N_{G_i}(K_n,e,f)}{n^{|V_i|-2}} \\
 &\leq  (1+{\eta}) \sum_{i=1}^s \beta_i (p^*)^{|E_i|-2}   2|E_i|(|E_i|-1).
\end{align*}
Applying this to \eqref{325}, by \eqref{323}, we have
\begin{align}\label{boundnorm12}
\sum_{e:e\neq f} U_{ef}(x) &\leq (1+{\eta})^2 \Phi'(\Psi_{\vec{\beta}} (p^*))    \sum_{i=1}^s 2\beta_i |E_i|(|E_i|-1) (p^*)^{|E_i|-2}\nonumber \\
&= (1+{\eta})^2   \Phi'(\Psi_{\vec{\beta}} (p^*))  \Psi'_{\vec{\beta}} (p^*)= (1+{\eta})^2 \varphi_{\vec{\beta}}'(p^*)<1-\delta.
\end{align}
This concludes the proof of \eqref{326}. 
 Therefore, by \eqref{324} and \eqref{326},  we have \eqref{205}.}

{
Equipped with \eqref{205}, one can conclude the proof of Proposition \ref{prop 4.3}.
Suppose that two configurations $y\leq z$ with $d_H(y,z)=\ell,$ differ at coordinates $i_1,\cdots,i_\ell$.  Then, consider a sequence
 $$y=w^0\leq w^1\leq \cdots \leq w^{\ell-1}\leq w^{\ell}=z,$$
 where $w^{k}=w^{k-1}\cup {i_k} .$  Since $y,z\in \mathcal{T}_\delta$ and $r_G(x,e)$ is an increasing function in $x$, we have $w^k\in \mathcal{T}_\delta$ for all $k=1,\cdots,l$. Note that using a  triangle inequality and the definition of Dobrushin Matrix \eqref{interdepend2},
\begin{align*} 
d_{TV}(\pi_{e}(\cdot |y),\pi_{e}(\cdot |z)) \leq \sum_{k=1}^{l} d_{TV}(\pi_{e}(\cdot |w^{k-1}),\pi_{e}(\cdot |w^{k}))\leq \sum_{k=1}^{l} a_{ei_k}(w^k).
\end{align*}
Hence,
\begin{align}  \label{225}
\mathbb{P}(Z^y_e(1)\neq Z^z_e(1)) &= \Big(1-\frac{1}{M}\Big)\1\{y_e\neq z_e\} + \frac{1}{M} d_{TV} (\pi_e(\cdot|y),\pi_e(\cdot |z)) \nonumber \\
& \leq \Big(1-\frac{1}{M}\Big)\1\{y_e\neq z_e\} + \frac{1}{M}  \sum_{k=1}^{l} a_{ei_k}(w^k).
\end{align}
Define a matrix $K'=(K'_{ef})_{e,f\in E(K_n)}$ in the following way: for any $e$ and $1\le k\le l,$ let $K'_{ei_k} :=  a_{e i_k}(w^k)$, and otherwise let   $K'_{ef}=0$ for $f\neq i_1,\cdots,i_l$. Since  $w^k\in  \mathcal{T}_\delta$, by \eqref{205}, we have $\norm{K'}_1  < 1-\delta$ for sufficiently small $\delta>0$. 
This is because $\norm{K'}_1$ is nothing but the maximum $L_1$ norm of its columns. Now, by definition,   columns of $K'$ indexed by edges other than $i_1,i_2,\ldots i_k$ are zero. Further, for any $1\le \ell \le k,$  the column of $K'$ indexed by $i_{\ell}$ agrees with the corresponding column of $A(w^\ell),$ which by \eqref{205}, has $L_1$ norm bounded by $1-\delta.$

Thus, the matrix
\begin{align*}
K=   \Big(1-\frac{1}{M}\Big) I + \frac{1}{M}
K'
\end{align*}
satisfies $\norm{K}_1< 1-\frac{\delta}{M}$. Now writing \eqref{225} for every edge $e$ in matrix form yields \eqref{431}.
}

\end{proof}

{
We next state a straightforward consequence of Proposition \ref{prop 4.3} which will be crucial in the proof of the Gaussian concentration result, which states that GD starting from configurations in the set $\mathcal{T}_\delta$ with small $\delta>0$ exhibits contraction. Formally, if $\vec{\beta}$ lies in the high temperature phase, then there exists sufficiently small $\delta>0$ such that for $y\leq z$ with $y,z\in \mathcal{T}_\delta$,  under the grand coupling,
\begin{align} \label{224}
\mathbb{E}[d_H(Z^y(1), Z^z(1))] \leq \Big(1-\frac{\delta}{M}\Big) d_H(y,z).
\end{align}

Therefore,  by \eqref{224} and Theorem \ref{theorem 2.1}, for any configurations $y\leq z$, under the grand coupling, for $t\geq cn^2$,
\begin{align*} 
\mathbb{E} [d_H(Z^y(t+1),Z^z(t+1))]  \leq \Big(1-\frac{\delta}{M}\Big) \mathbb{E} [d_H(Z^y(t),Z^z(t))]  + e^{-\Omega(n)}.
\end{align*}
It follows that for $t\geq cn^2$,
\begin{align} \label{235'}
   \mathbb{E} [d_H(Z^y(t),Z^z(t))] 
&\leq  \Big(1-\frac{\delta}{M}\Big)^{t-cn^2} \mathbb{E} [d_H(Z^y(cn^2),Z^z(cn^2))]   +  e^{-\Omega(n)}.
\end{align}}

In particular,
 recalling that $Z^+(t)$ and $Z^-(t)$ denote the GD starting from the complete and empty initial configurations respectively, {under the grand coupling}, for $t\geq cn^2$,
\begin{align} \label{235}
\mathbb{P}(Z^+(t)\neq Z^-(t)) \leq  \mathbb{E} [d_H(Z^+(t),Z^-(t))] 
&\leq  \Big(1-\frac{\delta}{M}\Big)^{t-cn^2}M +  e^{-\Omega(n)}.
\end{align}
Since ERGM with positive $\vec\beta$ is a monotone system,  by above, we have
\begin{align} \label{237}
\sup_{\mu,\nu \in \cP(\cG_n)} d_{TV} (\mu_t,\nu_t) \leq \Big(1-\frac{\delta}{M}\Big)^{t-cn^2}M +  e^{-\Omega(n)}.
\end{align}

The above preparation already allows us to finish the short  proof of the Poincar\'e inequality.
\begin{proof}[Proof of Theorem \ref{Theorem 0}] {The upper bound follows from the variation characterization \eqref{variational} and plugging in the test function $f: \cG_n \rightarrow \R$ where $f(x)=x_e-\E (X_e),$ where $e$ is a fixed edge and $X$ is distributed according to ERGM($\vec\beta$). To see this, note that for such a function, we have $(f(x^e) - f(x))^2=1$ and $(f(x^{e'})-f(x))^2 = 0$ for any $e'\neq e$, where recall from \eqref{generator} that $x^e$ denotes the configuration obtained from $x$ by flipping the state of $e.$ Thus, using $c(x,e)\le 1$, we have
  \begin{align} \label{200}
  \mathcal{E}(f,f) =  \frac{1}{M}\sum_{x\in \mathcal{G}_n} c(x,e) \pi(x) \leq \frac{1}{M} \leq \frac{C}{n^2}.
  \end{align}
In addition, by Remark \ref{remark 6.3} stated later, which gives the asymptotics $ \lim_{n \rightarrow \infty} \mathbb{P}(X_e=1)  = p^*$ in the high temperature regime,   one can deduce that $\int f^2 d\pi = \mathbb{E} (X_e^2 )-( \mathbb{E}X_e)^2\overset{n\to \infty}{\to}p^*(1-p^*)$ and hence is uniformly bounded away from $0$.
  This combined with \eqref{200}  gives a desired upper bound of order $\frac{1}{n^2}$ for the spectral gap.}
  
For the lower bound, by the following well known relation between {total variation distance} and spectral gap (see \cite[equation (12.15)]{lpw} for more details), 
$$
(1-\gamma)^{t}\le 2d_{TV}(Z^x(t),\pi),
$$
 it suffices to prove  that there exists $d>0$ such that for sufficiently large $n$ and  any $t\geq n^3$, there is a coupling of $Z^{+}(t)$ and $Z^{-}(t)$ satisfying
\begin{align} \label{252}
 \mathbb{P}(Z^{+}(t) \neq Z^{-}(t)) \leq e^{-d t/n^2}.
\end{align}
The proof is now complete by noticing that the above is a straightforward consequence of sub-multiplicative property of total variation distance and plugging in $t=n^3$ in \eqref{237}.
\end{proof}
In the next section, using the contraction estimates we prove the Gaussian concentration.

\section{Gaussian concentration: Proof of Theorem \ref{Theorem 2}} \label{section 7}
As indicated in Section \ref{iop},   Stein's method to prove concentration will be the key tool we rely on. We start by briefly reviewing the pertinent theory of the latter.
\subsection{Stein's method for  concentration} \label{section 7.1}
In \cite{c1}, Chatterjee introduced a beautiful new argument showing how Stein's method of exchangeable pairs can be used to obtain concentration results and presented several applications in \cite{c,cd2}. {For instance, a concentration result for the magnetization (say $m$) of the Curie-Weiss model with an inverse temperature $\beta$ and  external field $h$ was obtained  by showing, for any  $\beta\geq 0$, with high probability, $m \approx \tanh (\beta m + \beta h)$.}
Recall that a pair $(X,X')$ is said to be \it{exchangeable} if $(X,X')$ and $(X',X)$ have the same distribution.
The following is the key theorem we will use as input. 
\begin{theorem}\cite[{Theorem 1.5}]{c1}  \label{stein}
For a separable metric space $S$, let $(X,X')$ be an exchangeable pair of random variables taking values in $S$. Suppose that  $f:S\rightarrow \R$ and an antisymmetric function $F(x,y):S \times S \rightarrow \R$ are square-integrable, and satisfies
\begin{align} \label{311}
f(X):= \mathbb{E} [F(X,X') | X].
\end{align}
Then, define a function $g:S\rightarrow \R$ by
\begin{align} \label{713}
g(x) := \frac{1}{2} \mathbb{E} [|(f(X)-f(X'))F(X,X')|\big|X=x].
\end{align}
If $|g(x)|\leq C$, then for any $\theta \in \R$,
\begin{align}
 \mathbb{E} e^{\theta f(X)} \leq e^{C\theta^2 / 2}.
\end{align}
In particular, for any $t\geq 0$,
\begin{align}
\mathbb{P}(|f(X)| \geq t) \leq 2e^{-t^2 / 2C}.
\end{align}
\end{theorem}

We now describe the construction of the anti-symmetric function to be used. For an exchangeable pair $(X,X')$, there is a natural way to associate a reversible Markov kernel $\mathscr{P}$:
\begin{align} \label{711}
\sP f(x) =  \mathbb{E} [f(X')|X=x].
\end{align}
When we are given a function $f:S\rightarrow \R$ such that $ \mathbb{E} f(X)=0$ and 
\begin{align*}
\sup_{x,y\in S} \sum_{k=0}^\infty |\sP^kf(x)-\sP^kf(y)| <\infty,
\end{align*}
then it is easy to check that the function 
\begin{align} \label{712}
F(x,y):=\sum_{k=0}^\infty (\sP^kf(x)-\sP^kf(y))
\end{align}
is antisymmetric and satisfies the relation \eqref{311} (see \cite[Chapter 4]{c} for details).
We are now ready to prove Theorem \ref{Theorem 2}.

\begin{proof}[Proof of Theorem \ref{Theorem 2}]
Without loss of generality, let us assume that $ \mathbb{E}_{\pi} f=0$. Let $X=Z(0) \sim \pi$ and $X'=Z(1)$ be the step 1 distribution of GD starting from $X$. By reversibility of $P$ (see \eqref{711}), $(X,X')$ is an exchangeable pair. Let us define the antisymmetric function $F$ via   \eqref{712}, and subsequently the function $g$ as in  \eqref{713}.
The first order of business is to obtain an $L^\infty$ bound on the function $g$. We will in fact prove that \begin{align}\label{linf}
2|g(x)|  \leq C _0 \norm{\textup{\textbf{v}}}_1 \norm{\textup{\textbf{v}}}_\infty,
\end{align}
for some constant $C_0>0$ independent of $n$. This in conjunction with Theorem \ref{stein}, completes the proof of the theorem. The rest of the proof will be devoted to verifying \eqref{linf}.
By \eqref{713} 
\begin{align} \label{600}
2|g(x)|&= \mathbb{E} \left[\Big|(f(X)-f(X'))F(X,X')\Big||X=x\right] \nonumber \\
&{\leq} \frac{1}{M}\sum_{l=1}^M  \left|(f(x_{l+})-f(x_{l-}))F(x_{l+},x_{l-})\right|.
\end{align}
For each edge $l$, let us obtain an upper bound of $ |(f(x_{l+})-f(x_{l-}))F(x_{l+},x_{l-})|$.
For any configuration $x$ and edge $l$, let $X^{l,+}(t)=Z^{x_{l+}}(t)$ and $X^{l,-}(t)=Z^{x_{l-}}(t)$ be the GD starting from the initial configurations $x_{l+}$ and $x_{l-}$, respectively. Then, since $f$ is $\bf{v}$-Lipschitz,  we have 
\begin{align*}
 |f(x_{l+})-f(x_{l-})| |F(x_{l+},x_{l-})|&\leq v_l  \sum_{t=0}^\infty | \mathbb{E} f(X^{l,+}(t)) - \mathbb{E} f(X^{l,-}(t))|.
\end{align*}
Note that for each $t$ and any coupling of $X^{l,+}(t)$ and $X^{l,-}(t)$,
\begin{align*}
| \mathbb{E} f(X^{l,+}(t)) - \mathbb{E} f(X^{l,-}(t))| \leq \sum_{j=1}^M  v_j \mathbb{P} (X^{l,+}_j(t)\neq X^{l,-}_j(t)),
\end{align*} {
which implies that
\begin{align} \label{707}
 |f(x_{l+})-f(x_{l-})| |F(x_{l+},x_{l-})|\leq   v_l  \sum_{t=0}^\infty \sum_{j=1}^M  v_j \mathbb{P} (X^{l,+}_j(t)\neq X^{l,-}_j(t)).
\end{align} }
We will couple $X^{l,+}(t)$ and $X^{l,-}(t)$ through the {monotone coupling}. For such a coupling, let {$\textup{\textbf{r}}(x,l,t)$}  be the $M$-dimensional vector with $\textup{\textbf{r}}(x,l,t)_i := \mathbb{P}(X^{l,+}_i(t)\neq X^{l,-}_i(t))$.  Then,
\begin{align}\label{79}
\sum_{j=1}^M  v_j \mathbb{P} (X^{l,+}_j(t)\neq X^{l,-}_j(t)) \leq    \norm{\textup{\textbf{r}}(x,l,t)}_1 \norm{\textup{\textbf{v}}}_\infty  .
\end{align}
Since $\norm{L}_1$ { (recall that $L$ is defined in \eqref{crucialdef23})} is bounded above by a constant independent of $n$ (see \eqref{211}), using \eqref{contraction2} {and the fact $\|r(x,l,0)\|_1=1$}, there exists a constant $C>0$ such that for $ t\leq cn^2$, 
\begin{align*}
\norm{\textup{\textbf{r}}(x,l,t)}_1 \leq  \Big(1-\frac{1}{M}+\frac{\norm{L}_1}{M}\Big)^t \norm{\textup{\textbf{r}}(x,l,0)}_1 \leq C.
\end{align*}
Thus, it follows by \eqref{79} that
\begin{align} \label{703}
t\leq cn^2 \Rightarrow \sum_{j=1}^M  v_j \mathbb{P} (X^{l,+}_j(t)\neq X^{l,-}_j(t)) \leq C \norm{\textup{\textbf{v}}}_\infty.
\end{align}
Also, since $\norm{\textup{\textbf{r}}({x,l,cn^2})}_1\leq C$, by \eqref{235'} and the fact  
\begin{align*}
\mathbb{E}[d_H(X^{l,+}(t), X^{l,-}(t) )] = \mathbb{E} \Big[\sum_{j=1}^M \1 \{X^{l,+}_j(t) \neq X^{l,-}_j(t) \}\Big] =  \norm{\textup{\textbf{r}}(x,l,t)}_1,
\end{align*}
  for $cn^2< t< n^3 $,  
\begin{align} \label{702}
 \sum_{j=1}^M  v_j \mathbb{P} (X^{l,+}_j(t)\neq X^{l,-}_j(t))  &\leq  
\norm{\textup{\textbf{r}}(x,l,t)}_1 \norm{\textup{\textbf{v}}}_\infty   \leq  \Big(\big(1-\frac{\delta}{M}\big)^{t-cn^2}C +  e^{-\Omega(n)}\Big)\norm{\textup{\textbf{v}}}_\infty .
\end{align}

Now note that for $t\geq n^3$,
\begin{align}   \label{705}
\mathbb{P}(X^{l,+}(t)\neq X^{l,-}(t))\leq e^{-\alpha t/n^2}
\end{align}
for some $\alpha>0$.
{The above bound follows by plugging in $t=n^3$ in \eqref{235} and using the
sub-multiplicative property of the quantity  $\sup_{\mu,\nu \in \cP(\cG_n)} d_{TV} (\mu_t,\nu_t)$ along with monotonicity (see \eqref{237}).}  

By  \eqref{79} and \eqref{705}, 
\begin{align} \label{7011}
t\geq n^3 \Rightarrow \sum_{j=1}^M  v_j \mathbb{P} (X^{l,+}_j(t)\neq X^{l,-}_j(t)) \leq   n^2 e^{-\alpha t/n^2} \norm{\textup{\textbf{v}}}_\infty  \leq  e^{-\alpha' t/n^2} \norm{\textup{\textbf{v}}}_\infty
\end{align}
for some $\alpha'>0$ chosen slightly smaller than $\alpha$ to absorb the $n^{2}$ pre-factor.

Therefore,
applying  \eqref{703}, \eqref{702} and \eqref{7011}  to \eqref{707}, {
\begin{align*}
 |f(x_{l+})-f(x_{l-})|& |F(x_{l+},x_{l-})| \\
 &\leq v_l  \sum_{t=0}^\infty \sum_{j=1}^M  v_j \mathbb{P} (X^{l,+}_j(t)\neq X^{l,-}_j(t)) \\
 &\leq v_l  \norm{\textup{\textbf{v}}}_\infty  \Big( \sum_{t=0}^{cn^2}  C  +  \sum_{t=cn^2+1}^{n^3-1}  \Big(\big(1-\frac{\delta}{M}\big)^{t-cn^2}C +  e^{-\Omega(n)}\Big) +  \sum_{t=n^3}^{\infty}   e^{-\alpha' t/n^2} \Big) \\
 &\leq v_l   \norm{\textup{\textbf{v}}}_\infty  O(n^2).
\end{align*} 
 Recalling that  $M$ is of order $n^2$,}
this implies that for some constant $C_0>0$, 
\begin{align*}
2|g(x)| \leq \frac{1}{M}\sum_{l=1}^M |f(x_{l+})-f(x_{l-})| |F(x_{l+},x_{l-})|\leq C_0 \norm{\textup{\textbf{v}}}_1 \norm{\textup{\textbf{v}}}_\infty,
\end{align*}
which in conjunction with Theorem \ref{stein}, implies that for  $t\geq 0$,
\begin{align} \label{710}
\mathbb{P}(|f(X)| \geq t) \leq 2e^{- t^2 /  C_0 \norm{\textup{\textbf{v}}}_1 \norm{\textup{\textbf{v}}}_\infty}.
\end{align} 
\end{proof}

\section{Central limit theorem: Proof of Theorem \ref{Theorem 3}} \label{section 8}

The proof is at a very high level based on the moment method and will rely on delicate and novel estimates of the $k$-correlation functions of the form, $\mathbb{E} [(X_{i_1}- \mathbb{E} X_{i_1})\cdots (X_{i_k}- \mathbb{E} X_{i_k})] $ for edges $i_1,i_2,\ldots i_k$.  We first start with a few related results that will be used.
\begin{lemma} For any fixed positive integer $k$ and distinct edges $i_1,\cdots,i_k$, 
\begin{align} \label{720}
|\mathbb{E} [X_{i_1}\cdots X_{i_k}]  - (\mathbb{E}X_{i_1})\cdots(\mathbb{E}X_{i_k}) | = O\Big(\frac{1}{n}\Big).
\end{align}
\end{lemma}
\begin{proof} The case $k=2$ easily follows from the Gaussian concentration. {In fact, since the Lipschitz vector of the function $X_1+\cdots+X_M$ is $(1,\cdots,1)$,
Theorem \ref{Theorem 2} implies that }
\begin{align} \label{gaussian}
\text{Var}(X_1+\cdots+X_M) = O(M).
\end{align}
By symmetry, one can deduce that 
if two edges $i$ and $j$ does not share a vertex, then
\begin{align} \label{612}
\text{\text{\text{Cov}}}(X_i,X_j) = O\Big(\frac{1}{n^2}\Big),
\end{align}
whereas if two {distinct} edges $i$ and $j$  share a vertex, then
\begin{align} \label{613}
\text{\text{Cov}}(X_i,X_j) = O\Big(\frac{1}{n}\Big).
\end{align}
Above we use the fact that $\text{Cov}(X_i,X_j)\geq 0$ by FKG inequality.
In the case $k\geq 3$, we will crucially rely on the following estimate.

\noindent
\textbf{Fact} \cite[Equation (12)]{n}: {If random variables $(Z_1,\cdots,Z_m)$ satisfy the positive  quadrant dependent condition, i.e., for all $x,y \in \R,$ }
\begin{align} \label{positive  quadrant}
\mathbb{P}(Z_i > x, Z_j > y) \geq \mathbb{P}(Z_i > x)\mathbb{P}(Z_j > y),
\end{align}then there exists a constant $C>0$ such that for any {$C^1$ functions} $f,g$ with bounded partial derivatives,
\begin{align} \label{619}
|\text{Cov}(f(Z_1,\cdots,Z_m),g(Z_1,\cdots,Z_m))| \leq C \sum_{i=1}^m \sum_{j=1}^m \norm{\frac{\partial f}{\partial z_i}}_\infty  \norm{\frac{\partial g}{\partial z_j}}_\infty \text{Cov}(Z_i,Z_j).
\end{align}
{ Since $(X_{i_1},\cdots,X_{i_j})$ satisfy \eqref{positive  quadrant} by FKG inequality,}
 using \eqref{619} with \eqref{612} and \eqref{613}, 
\begin{align*}
|\text{Cov}(X_{i_1}\cdots X_{i_{j-1}},X_{i_j}) |=  O\Big(\frac{1}{n}\Big)
\end{align*}
for any $j$ {(this can be done by choosing $g(x)=x$ and $f(x_1,\cdots,x_{j-1}) = \varphi(x_1,\cdots,x_{j-1})  x_1\cdots x_{j-1}$ with a compactly supported smooth function $\varphi$ satisfying $\varphi=1$ on $[-2,2]^{j-1}$)}. Using this fact repeatedly for $1\le j\le k$,  we obtain \eqref{720}. 
\end{proof}

Next, we derive the following quantitative independence result, which will be a key ingredient to obtain $k$-correlation estimates.

\begin{proposition} \label{prop 5.1}
For any fixed positive integer $k$ and  $a_1,\cdots,a_k,b_1,\cdots,b_k \in \{0,1\}$, the following statement holds. If the edge $j$ does not share a vertex with edges $i_1,\cdots,i_k$, then
\begin{align}  \label{610}
\Big\vert \mathbb{P}(X_{j} = 1 | X_{i_1}=a_1,\cdots,X_{i_k}=a_k) -  \mathbb{P}(X_{j} = 1 | X_{i_1}=b_1,\cdots,X_{i_k}=b_k) \Big\vert= O\Big(\frac{1}{n^2}\Big).
\end{align}
Whereas,  if the edge $j$  shares  vertices with some of the edges $i_1,\cdots,i_k$, then
\begin{align}  \label{611}
\Big\vert\mathbb{P}(X_{j} = 1 | X_{i_1}=a_1,\cdots,X_{i_k}=a_k) -  \mathbb{P}(X_{j} = 1 | X_{i_1}=b_1,\cdots,X_{i_k}=b_k) \Big\vert= O\Big(\frac{1}{n}\Big).
\end{align}

\end{proposition}

\begin{remark} \label{remark 6.3}
A qualitative version of Proposition \ref{prop 5.1} is obtained in  
\cite[Thoerem 7]{bbs} which states that if ${\vec{\beta}}$ lies in the high temperature regime, and  $p^*$ is the unique maximizer of the variational formula \eqref{variational}, then, for any fixed positive integer $k$,
\begin{align}  \label{independent}
\lim_{n\rightarrow \infty} \mathbb{P}(X_1=a_1,\cdots,X_k=a_k) = (p^*)^{\sum_i a_i}(1-p^*)^{k-\sum_i a_i}, 
\end{align}
but does not provide any fine information about the covariance structure.
\end{remark}

\begin{proof}[Proof of Proposition \ref{prop 5.1}]
Let us first consider the case when  the edge $j$ does not share a vertex with edges $i_1,\cdots,i_k$. In order to prove \eqref{610}, it suffices to check that
\begin{align} \label{614}
\Big\vert \mathbb{P}(X_{j} = 1 | X_{i_1}=a_1,\cdots,X_{i_k}=a_k) -  \mathbb{P}(X_{j} = 1) \Big\vert = O\Big(\frac{1}{n^2}\Big).
\end{align}
Due to \eqref{independent}, it reduces to prove that
\begin{align} \label{615}
\Big\vert\mathbb{P}(X_{j} = 1, X_{i_1}=a_1,\cdots,X_{i_k}=a_k) -  \mathbb{P}(X_{j} = 1)\mathbb{P}(X_{i_1}=a_1,\cdots,X_{i_k}=a_k) \Big\vert= O\Big(\frac{1}{n^2}\Big).
\end{align}
Let $l$ be the number of 0's in $a_1,\cdots,a_k$. Let us prove \eqref{615} by the induction on $l$. Suppose that $l=0$, in  other words, $a_1=\cdots=a_k=1$. Then, the LHS of \eqref{615} is equal to
\begin{align*}
\text{Cov}(X_j,X_{i_1}\cdots X_{i_k}).
\end{align*}
{Since $(X_{i_1},\cdots,X_{i_k},X_j)$ satisfies the condition \eqref{positive  quadrant} due to FKG inequality,} applying \eqref{619}  to $\text{Cov}(X_j,X_{i_1}\cdots X_{i_k})$ and using \eqref{612},
\begin{align*}
|\text{Cov}(X_j,X_{i_1}\cdots X_{i_k})| \leq C( \text{Cov}(X_j,X_{i_1})+\cdots+ \text{Cov}(X_j,X_{i_k})) = O\Big(\frac{1}{n^2}\Big).
\end{align*}
This concludes the proof of \eqref{615} when $l=0$.
Suppose that \eqref{615} holds for $l$, and let us prove it for $l+1$. Without loss of generality, assume that $a_k=0$. Then,
\begin{align*}
&\mathbb{P}(X_{j} = 1, X_{i_1}=a_1,\cdots,X_{i_k}=a_k) -  \mathbb{P}(X_{j} = 1)\mathbb{P}(X_{i_1}=a_1,\cdots,X_{i_k}=a_k) \\
&=  \mathbb{P}(X_{j} = 1, X_{i_1}=a_1,\cdots,X_{i_{k-1}}=a_{k-1}) - \mathbb{P}(X_{j} = 1, X_{i_1}=a_1,\cdots,X_{i_{k-1}}=a_{k-1},X_{i_k}=1) \\
&\qquad -  \mathbb{P}(X_{j} = 1) ( \mathbb{P}(X_{i_1}=a_1,\cdots,X_{i_{k-1}}=a_{k-1})- \mathbb{P}(X_{i_1}=a_1,\cdots,X_{i_{k-1}}=a_{k-1},X_{i_k}=1)).
\end{align*}
By the induction hypothesis, the absolute value of the above expression is $O(\frac{1}{n^2})$. This concludes  the proof of \eqref{615} for general $l$.
The proof of \eqref{611}  when the edge $j$  shares  vertices with some of the edges $i_1,\cdots,i_k$ is  same as above, except  that
\begin{align*}
|\text{Cov}(X_j,X_{i_1}\cdots X_{i_k})| \leq C( \text{Cov}(X_j,X_{i_1})+\cdots+ \text{Cov}(X_j,X_{i_k})) = O\Big(\frac{1}{n}\Big)
\end{align*}
by \eqref{613} and \eqref{619}.
\end{proof}

Using the quantitative  independence result Proposition \ref{prop 5.1} and  Theorem \ref{Theorem 2}, we obtain the following estimate on the $k$-correlation $\mathbb{E} [(X_{i_1}- \mathbb{E} X_{i_1})\cdots (X_{i_k}- \mathbb{E} X_{i_k})] $.
\begin{proposition} \label{prop 5.2}
{
 For any given integers $l,m \geq 0$ and $a_1,\cdots,a_m\geq 1$ satisfying $a_1+\cdots+a_m\leq 2m$, there exists a constant $C_{l,a_1,\cdots,a_m}>0$ such that the following statement holds for sufficiently large $n$: for any set of  edges   $i_1,\cdots,i_m,j_1,\cdots,j_l$  not sharing a vertex, 
\begin{align} \label{621}
\big\vert  \mathbb{E} [(X_{i_1}- \mathbb{E} X_{i_1})^{a_1}\cdots (X_{i_m}- \mathbb{E} X_{i_m})^{a_m} | X_{j_1},\cdots,X_{j_l}]\big\vert \leq \frac{C_{l,a_1,\cdots,a_m}}{n^{m-(a_1+\cdots+a_m)/2}}.
\end{align}
In particular, for any given positive integer $k$, for $k$ different edges $i_1,\cdots,i_k$ who do not share a vertex,
\begin{align} \label{620}
\big\vert \mathbb{E} [(X_{i_1}- \mathbb{E} X_{i_1})\cdots (X_{i_k}- \mathbb{E} X_{i_k})] \big\vert= O\Big(\frac{1}{n^{k/2}}\Big).
\end{align}
}
\end{proposition}

\begin{proof}
{
It suffices to simply prove  \eqref{621}, since \eqref{620} directly follows by taking $m=k, l=0, a_1=a_2\ldots=a_m=1$.}
Define $\widetilde{X}_i:=X_i- \mathbb{E} X_i $, and let us  prove \eqref{621} by the induction on $m$. Without loss of generality, we assume that $j_1=1,\cdots,j_l=l$, $\{i_1,\cdots,i_m\} \subset \{l+1,\cdots,[n/2]\}$, and the edges $1,\cdots, [n/2]$ do not share a vertex. The case when $m=1$ is obvious. In fact, according to Proposition \ref{prop 5.1},  for any fixed $l\geq 0$,
\begin{align*}
\big\vert \mathbb{E} [\widetilde{X}_{i_1} | X_1,\cdots,X_l] \big\vert= O\Big(\frac{1}{n}\Big),
\end{align*}
which implies \eqref{621} for $a_1=1$. Also, it is obvious that $|\widetilde{X}_{i_1}|\leq 1$, which implies \eqref{621} in the case $a_1=2$.

Now suppose that  \eqref{621} is true for $m\leq k-1$, and let us prove it for $m=k$. First, consider the case when at least one of $a_1,\cdots,a_k$ is greater that one. Assuming that $a_1\geq 2$, we have
\begin{align} \label{622}
 \mathbb{E}[\widetilde{X}_{i_1}^{a_1}\cdots \widetilde{X}_{i_k}^{a_k} | X_1,\cdots,X_l] =   \mathbb{E}\big[\widetilde{X}_{i_1}^{a_1} \mathbb{E}[\widetilde{X}_{i_2}^{a_2}\cdots \widetilde{X}_{i_k}^{a_k} | X_1,\cdots,X_l,X_{i_1}] | X_1,\cdots,X_l\big].
\end{align}
Since $a_2+\cdots+a_k\leq 2(k-1)$,  by the induction hypothesis,
\begin{align} \label{623}
\big\vert  \mathbb{E}[\widetilde{X}_{i_2}^{a_2}\cdots \widetilde{X}_{i_k}^{a_k} | X_1,\cdots,X_l,X_{i_1}] \big\vert \leq  \frac{C_{l+1,a_2,\cdots,a_k}}{n^{m-1-(a_2+\cdots+a_k)/2}} \leq \frac{C_{l+1,a_2,\cdots,a_k}}{n^{m-(a_1+\cdots+a_k)/2}}.
\end{align}
Since $|\widetilde{X}_{i_1}|\leq 1$, \eqref{622} and \eqref{623} conclude the proof of \eqref{621}.
Now, let us consider the case when $a_1=\cdots=a_k=1$. In other words, let us prove that for any fixed $l\geq 0$,
\begin{align} \label{624}
\big\vert \mathbb{E} [\widetilde{X}_{i_1}\cdots \widetilde{X}_{i_k} | X_1,\cdots,X_l]\big\vert = O\Big(\frac{1}{n^{k/2}}\Big).
\end{align}
According to the Gaussian concentration  Theorem \ref{Theorem 2}, 
\begin{align} \label{625}
\mathbb{P}(|\widetilde{X}_{l+1} + \cdots + \widetilde{X}_{[n/2]}|>t)\leq 2 \exp\Big(-\frac{ct^2}{[n/2]-l}\Big)
\end{align}
{(Since the Lipschitz vector of $X_{l+1}+\cdots+X_{[n/2]}$ is simply a vector with $1$ corresponding to the edges $l+1$ to $[n/2]$ and $0$ everywhere else, the $L^1$ and $L^\infty$ norms are $[n/2]-l$ and $1$ respectively).}
Note that due to \eqref{independent}, for any  $l\geq 0$, there exists a constant $C_l>0$ such that for sufficiently large $n$ and any event $\mathcal{A}$, 
\begin{align*}
\mathbb{P}(\mathcal{A}| X_1,\cdots,X_l) \leq C_l \mathbb{P}(\mathcal{A}).
\end{align*} 
Combining this  with \eqref{625}, for any fixed integers $l\geq 0$ and $k\geq 1$, we have the conditional concentration result
\begin{align} \label{626}
\Big\vert \mathbb{E} \big[(\widetilde{X}_{l+1} + \cdots + \widetilde{X}_{[n/2]})^k | X_1,\cdots,X_l\big] \Big\vert= O(n^{k/2}).
\end{align}
Expanding left-hand-side of \eqref{626}, we have terms like $\mathbb{E} [\widetilde{X}_{b_1}^{c_1}\cdots \widetilde{X}_{b_L}^{c_L}|X_1,\cdots,X_l] $ with $c_1\geq \cdots \geq c_L\geq 1$ and $c_1+\cdots+c_L=k$. If $c_1\geq 2$ and $k=c_1+\cdots+c_L\leq 2L$, then since $L\leq k-1$, by the induction hypothesis,
\begin{align*}
\big\vert \mathbb{E} [\widetilde{X}_{b_1}^{c_1}\cdots \widetilde{X}_{b_L}^{c_L}|X_1,\cdots,X_l]\big\vert\leq   \frac{C_{l,c_1,\cdots,c_L}}{n^{L-(c_1+\cdots+c_L)/2}}=\frac{C_{l,c_1,\cdots,c_L}}{n^{L-k/2}}.
\end{align*}
Since the number of such terms is less than $n^L$, the sum of such terms is bounded by  $C_{l,c_1,\cdots,c_L} n^{k/2}$.

Whereas, if $k=c_1+\cdots+c_L \geq 2L+1$, then each term $|\mathbb{E} [\widetilde{X}_{b_1}^{c_1}\cdots \widetilde{X}_{b_L}^{c_L}|X_1,\cdots,X_l] |$ is bounded by 1, and thus sum of such terms is bounded by $n^L\leq n^{(k-1)/2}$.
Therefore, applying the aforementioned facts to \eqref{626}, one can conclude that the absolute value of the sum of terms {$\mathbb{E} [\widetilde{X}_{b_1}\cdots \widetilde{X}_{b_k}|X_1,\cdots,X_l] $}  is bounded by $O(n^{k/2})$. {Since there are $\Theta(n^k)$ many such terms, we obtain \eqref{624} by symmetry which implies they have the same values.}
\end{proof}

{
\begin{remark}
Note that in  Proposition \ref{prop 5.2}, we only obtained  higher order correlation estimates for edges not sharing a vertex.  The reason behind this is as follows.
Recall that FKG inequality, which ensures the non-negativity of pairwise-correlations, played a crucial role in obtaining the estimates \eqref{612} and \eqref{613} using the Gaussian concentration estimate \eqref{gaussian}
and the underlying symmetry. 
However,  in general, one cannot ensure the positivity of  higher order correlations just under a monotone measure. This is the main obstacle we face.

To elaborate on this, note first that the Gaussian concentration estimate yields the following for any $m\le M={n\choose{2}}$:
   \begin{align} \label{concentration}
   \mathbb{E} | \tilde{X}_1+ \cdots+\tilde{X}_m|^k \leq Cm^{k/2}.
   \end{align}
   Expanding the left hand side, we obtain various correlation terms of the form $\E(\tilde{X}^{a_1}_{i_1}\tilde{X}^{a_2}_{i_2}\ldots\tilde{X}^{a_\ell}_{i_\ell})$. Since they are not all guaranteed to be positive, apriori we cannot rule out large correlations with different signs across different monomials canceling each other out. 
   
Further, even terms corresponding to a particular type of a monomial are not expected to have the same value, and will depend on the graph structure induced by the edges $i_1, i_2, i_3, \ldots i_\ell$.   
   
This forces us to work with disjoint edges so that there is a single graph structure for any monomial. This means we can take $m$ to be at most $O(n).$ In this case, in \eqref{concentration}, the expansion yields $n^{k}$ terms of the form $\E(\tilde{X}_{i_1}\tilde{X}_{i_2}\ldots\tilde{X}_{i_k})$ and the right hand side is $O(n^{k/2}).$ This allows us, using induction and symmetry, to prove the $O(\frac{1}{n^{k/2}})$ bound on the $k-$point correlation. The above discussion also indicates that to improve on this, one needs to quantify the correlation depending on the graph structure induced by the corresponding edges, which is significantly more delicate. 
\end{remark}}

Given the above preparation we can now finish the proof of Theorem \ref{Theorem 3}.
\begin{proof} [Proof of Theorem \ref{Theorem 3}]
Let $S_m:= \frac{X_{i_1} + \cdots+X_{i_m}  -  \mathbb{E} [X_{i_1} + \cdots+X_{i_m} ]}{\sqrt{\text{Var}(X_{i_1} + \cdots+X_{i_m})}}$ and $\widetilde{X}_i:=X_i- \mathbb{E} X_i $. Throughout the proof, we fix a positive integer $k$, and let us compute the $k$-th moment of $\widetilde{X}_{i_1} + \cdots + \widetilde{X}_{i_m}$:
\begin{align*}
 \mathbb{E}\big[(\widetilde{X}_{i_1} + \cdots + \widetilde{X}_{i_m})^k\big] = \sum_{a_1,\cdots,a_m\geq 0, \ a_1+\cdots+a_m=k} {k \choose a_1,\cdots,a_m}  \mathbb{E} \widetilde{X}_{i_1}^{a_1}\cdots \widetilde{X}_{i_m}^{a_m}.
\end{align*}
For $b_1\geq b_2\geq \cdots \geq b_m\geq 0$ satisfying $b_1+\cdots+b_m=k$,
define $A(b_1,\cdots,b_m)$ to be a collection of $(a_1,\cdots,a_m)$s whose non-increasing rearrangement is $(b_1,\cdots,b_m)$.
 We claim that unless $$b_1=\cdots=b_{k/2}=2, b_{k/2 +1} = \cdots=b_m=0,$$ 
\begin{align} \label{633}
\sum_{(a_1,\cdots,a_m)\in A(b_1,\cdots,b_m)} { {k \choose a_1,\cdots,a_m} \Big\vert  \mathbb{E}\big[ \widetilde{X}_{i_1}^{a_1}\cdots \widetilde{X}_{i_m}^{a_m} \big] \Big\vert } = o(m^{k/2}).
\end{align}
Note that $k/2$ is not an integer for $k$ odd, which means that  \eqref{633} holds for any $b_1,\cdots,b_m$ when $k$ is odd.
Let $f(b_1,\cdots,b_m)$  be the number of non-zero values in $b_1,\cdots,b_m$.
According to \eqref{621}, if $f(b_1,\cdots,b_m)> k/2$, then for each $(a_1,\cdots,a_m)\in A(b_1,\cdots,b_m)$,
\begin{align*}
\Big\vert  \mathbb{E} \big[\widetilde{X}_{i_1}^{a_1}\cdots \widetilde{X}_{i_m}^{a_m}\big] \Big\vert = O\Big(\frac{1}{n^{f(b_1,\cdots,b_m)-k/2}}\Big).
\end{align*}
Since $|A(b_1,\cdots,b_m)| \leq m^{f(b_1,\cdots,b_m)}$ and $m=o(n)$, one can conclude that if $f(b_1,\cdots,b_m)> k/2$, then
\begin{align} \label{634}
\sum_{(a_1,\cdots,a_m)\in A(b_1,\cdots,b_m)} \Big\vert \mathbb{E}\big[ \widetilde{X}_{i_1}^{a_1}\cdots \widetilde{X}_{i_m}^{a_m}\big] \Big\vert = O\Big(\frac{m^{f(b_1,\cdots,b_m)}}{n^{f(b_1,\cdots,b_m)-k/2}}\Big) = o(m^{k/2}).
\end{align}
Since $\mathbb{E} |\widetilde{X}_{i_1}^{a_1}\cdots \widetilde{X}_{i_m}^{a_m}|\leq 1$, 
if $f(b_1,\cdots,b_m)<k/2$, then
\begin{align} \label{635}
\sum_{(a_1,\cdots,a_m)\in A(b_1,\cdots,b_m)} \Big\vert \mathbb{E}\big[ \widetilde{X}_{i_1}^{a_1}\cdots \widetilde{X}_{i_m}^{a_m}\big] \Big\vert = O(m^
{f(b_1,\cdots,b_m)}) = o(m^{k/2}).
\end{align}
In the case when $k$ is even and $f(b_1,\cdots,b_m)=k/2$, we prove that unless $b_1=\cdots=b_{k/2}=2$ and $b_{k/2 +1} =\cdots=b_m=0$, for $(a_1,\cdots,a_m)\in A(b_1,\cdots,b_m)$,
\begin{align} \label{636}
\Big\vert \mathbb{E} \big[\widetilde{X}_{i_1}^{a_1}\cdots \widetilde{X}_{i_m}^{a_m}\big] \Big\vert = O\Big(\frac{1}{n^{1/2}}\Big).
\end{align}
We have $b_1\geq 3$  unless $b_1=\cdots=b_{k/2}=2$ and $b_{k/2 +1} =\cdots=b_m=0$.
Without loss of generality, assume that $a_j=b_j$ for $j=1,2,\cdots,m$. Note that
\begin{align} \label{637}
 \mathbb{E}\big[ \widetilde{X}_{i_1}^{a_1}\cdots \widetilde{X}_{i_m}^{a_m}\big] =  \mathbb{E}  \big[\widetilde{X}_{i_1}^{a_1}  \mathbb{E}[\widetilde{X}_{i_2}^{a_2}\cdots\widetilde{X}_{i_m}^{a_m} | X_{i_1} ]\big].
\end{align}
Since $a_1+\cdots+a_m=k$ and $a_1\geq 3$, we have $a_2+\cdots+a_m \leq k-3$. Therefore, by \eqref{621}, we have 
\begin{align} \label{638}
\Big\vert  \mathbb{E} [\widetilde{X}_{i_2}^{a_2}\cdots\widetilde{X}_{i_m}^{a_m} | X_{i_1} ]  \Big\vert = O\Big(\frac{1}{n^{k/2 - 1 - (a_2+\cdots+a_m)/2}}\Big) = O\Big(\frac{1}{n^{1/2}}\Big).
\end{align}
 Therefore, \eqref{637}
 and \eqref{638} concludes the proof of  \eqref{636}. Using the fact $|A(b_1,\cdots,b_m)| \leq m^{f(b_1,\cdots,b_m)}$, \eqref{636} implies that  when $k$ is even and $f(b_1,\cdots,b_m)=k/2$,  unless $b_1=\cdots=b_{k/2}=2$ and $b_{k/2 +1} =\cdots=b_m=0$,
 \begin{align} \label{801}
 \sum_{(a_1,\cdots,a_m)\in A(b_1,\cdots,b_m)} \Big\vert  \mathbb{E}\big[ \widetilde{X}_{i_1}^{a_1}\cdots \widetilde{X}_{i_m}^{a_m}\big] \Big\vert =  O\Big(\frac{m^{k/2}}{n^{1/2}}\Big) = o(m^{k/2}).
 \end{align}
Since the multinomial coefficient ${k \choose a_1,\cdots,a_m} $ is always bounded by $k!$, we obtain \eqref{633} by \eqref{634}, \eqref{635}, and \eqref{801}.

Using \eqref{633}, one concludes that if $k$ is odd, then
\begin{align} \label{639}
{\Big\vert} \mathbb{E}[(\widetilde{X}_{i_1} + \cdots + \widetilde{X}_{i_m})^k] {\Big\vert} = o(m^{k/2}).
\end{align}
Whereas, if $k$ is even, then
\begin{align} \label{6300}
 \mathbb{E}[(\widetilde{X}_{i_1} + \cdots + \widetilde{X}_{i_m})^k] = o(m^{k/2}) + {k \choose 2,\cdots,2}  \sum_{j_1<\cdots<j_{k/2}}   \mathbb{E} \widetilde{X}_{i_{j_1}}^2\cdots \widetilde{X}_{i_{j_{k/2}}}^2.
\end{align}
Note that by  symmetry,
\begin{align} \label{6301}
\text{Var}(X_{i_1}+\cdots+X_{i_m}) = m\text{Var}(X_{i_1}) + m(m-1)\text{Cov}(X_{i_1},X_{i_2}).
\end{align}
Since the law of $X_{i_1}$ converges weakly to the Bernoulli distribution with a parameter $p^*$,
\begin{align*}
\lim_{n\rightarrow \infty}\text{Var}(X_{i_1})  = p^*(1-p^*).
\end{align*}
Recall that  $\text{Cov}(X_{i_1},X_{i_2}) = O(\frac{1}{n^2})$ since $i_1$ and $i_2$ do not share a vertex (see \eqref{612}). Thus, applying this to \eqref{6301}, 
\begin{align} \label{6304}
\lim_{n\rightarrow \infty} \frac{1}{m}\text{Var}(X_{i_1}+\cdots+X_{i_m}) =p^*(1-p^*).
\end{align}
Therefore, by \eqref{639} and \eqref{6304}, if $k$ is odd, then
\begin{align} \label{6303}
\lim_{n\rightarrow \infty}  \mathbb{E} S_m^k= 0.
\end{align}
Whereas, if $k$ is even, then by \eqref{6300} and symmetry,
\begin{align} \label{6305}
 \mathbb{E} S_m^k= o(1) + {k \choose 2,\cdots,2} {m \choose k/2} \frac{  \mathbb{E} \widetilde{X}_{i_1}^2\cdots \widetilde{X}_{i_{k/2}}^2}{(\text{Var}(X_{i_1}+\cdots+X_{i_m}) )^{k/2}}.
\end{align}
According to the asymptotic independence result \eqref{independent}, we have
\begin{align} \label{6306}
\lim_{n\rightarrow \infty}  \mathbb{E} \big[\widetilde{X}_{i_1}^2\cdots \widetilde{X}_{i_{k/2}}^2  \big]=  (p^*(1-p^*))^{k/2}.
\end{align}
Therefore, applying \eqref{6304} and \eqref{6306} to \eqref{6305}, one can conclude that for $k$ even,
\begin{align} \label{6302}
\lim_{n\rightarrow \infty}   \mathbb{E} S_m^k = (k-1)!!
\end{align}
Here, $(k-1)!!$ denotes the product of all numbers from 1 to $k-1$ that have the same parity as $k-1$. Thus,
\eqref{6303} and \eqref{6302} imply that the normalized sum $S_m$ weakly converges to the standard normal distribution, by, say, the classical Fr\'echet-Shohat theorem \cite{clt} along with the fact that the normal distribution is a  determinate measure, i.e., its moments determine it uniquely. 
\end{proof}

\section{Bounding $W_1$-Wasserstein distance: Proof of Theorem \ref{w1distance}}\label{section 9}
We start with the following related result about one point marginals.
\begin{proposition} \label{prop 7.3} For any edge $i$,
\begin{align} \label{640}
| \mathbb{E} X_i - p^* | = O\Big(\frac{\sqrt{\log n}}{\sqrt{n}}\Big).
\end{align}
\end{proposition}
\begin{proof} By symmetry, let us consider the case $i=1$. Note that
\begin{align} \label{642}
\mathbb{E}  X_1 = \mathbb{E} \big[\mathbb{E}[X_1 | X_2,\cdots,X_M]\big] = \mathbb{E}   { \Phi(  \partial_1 H(X)  ) },
\end{align}
{where in the last inequality we used 
\begin{align*}
\mathbb{E}[X_1 | X_2,\cdots,X_M] = \mathbb{P}(X_1 = 1 | X_2,\cdots, X_M) =  \Phi(  \partial_1 H(X) ).
\end{align*}}
 The Lipschitz constant of $\partial_1 H(x)$ with respect to the edge $j$ is $O(\frac{1}{n^2})$ if $j$ does not share a vertex with the edge $1$ and $O(\frac{1}{n})$ otherwise. To see this,  note first that by \eqref{802},
 \begin{align*}
 \partial_1 H(x)  = \sum_{i=1}^s \beta_i \frac{  N_{G_i}(x,1)}{n^{|V_i|-2}}.
\end{align*}  
 As mentioned in the proof of  Proposition \ref{prop 4.3}, the quantity $N_{G_i}(x_{f+},1) - N_{G_i}(x_{f-},1) = N_{G_i}(x,1,f)$ is  $O(n^{|V_i|-4})$ if $f$ does not share a vertex with the edge $1$ and $O(n^{|V_i|-3})$ if $f$ shares a vertex with the edge $1$. Due to the normalization factor $n^{|V_i|-2}$, we have the asserted bound on the entries of the Lipschitz vector for $\partial_1 H(x)$.
 
Hence, its $L^1$ and $L^\infty$ norms are  $O(1)$ and $O(\frac{1}{n})$ respectively.  Thus, according to Theorem \ref{Theorem 2},  for some constant $c>0$, for all $\e>0$
\begin{align}  \label{641}
\mathbb{P}(|\partial_1 H(X) - \mathbb{E} \partial_1 H(X)|>\varepsilon) \le 2\exp(-cn\varepsilon^2)
\end{align}
For $K>0$ chosen later, let us denote
\begin{align*}
A_n:=\mathbb{E} \Big[  { \Phi(  \partial_I H(X)  ) }  : |\partial_1 H(X) - \mathbb{E} \partial_1 H(X)|\leq \frac{K}{\sqrt{n}}\Big], \\
B_n:=\mathbb{E} \Big[ { \Phi(  \partial_I H(X)  ) }  : |\partial_1 H(X) - \mathbb{E} \partial_1 H(X)|>\frac{K}{\sqrt{n}}\Big].
\end{align*}
Then, by \eqref{642}, we have
\begin{align} \label{643}
\mathbb{E}  X_1   = A_n+B_n.
\end{align}
{Since $\Phi(x) = \frac{e^x}{1+e^x}$ is non-negative and increasing,
\begin{align} \label{644}
 { \Phi\Big(   \mathbb{E} \partial_I H(X)  - \frac{K}{\sqrt{n}} \Big) }   
 \mathbb{P}\Big(|\partial_1 H(X) - \mathbb{E} \partial_1 H(X)|\leq \frac{K}{\sqrt{n}}\Big)   \leq A_n \leq  { \Phi\Big(   \mathbb{E} \partial_I H(X)  + \frac{K}{\sqrt{n}}\Big ) } .
\end{align}
Choosing $\e=\frac{K}{\sqrt n}$ in \eqref{641}, and   using the fact that $\Phi(x) \leq 1$,
\begin{align} \label{644}
 { \Phi\Big(   \mathbb{E} \partial_I H(X)  - \frac{K}{\sqrt{n}} \Big) }   
 - 2\exp(-cK^2)  \leq A_n \leq  { \Phi\Big(   \mathbb{E} \partial_I H(X)  + \frac{K}{\sqrt{n}}\Big ) }.
\end{align}
Next, using $\Phi(x)\le 1$ and  \eqref{641} again, we obtain: }
\begin{align} \label{645}
0\leq B_n \leq \mathbb{P}\Big(|\partial_1 H(X) - \mathbb{E} \partial_1 H(X)|>\frac{K}{\sqrt{n}}\Big)\le 2\exp(-cK^2).
\end{align}
Now letting $\mathbb{E} X_1 = a_n$, we express $\mathbb{E} \partial_1 H(X)$ in terms of $a_n$. {By \eqref{802},
\begin{align*}
\mathbb{E} \partial_1 H(X)  = \sum_{i=1}^s \beta_i \frac{\mathbb{E} N_{G_i}(X,1)}{n^{|V_i|-2}},
\end{align*}
and note that} by \eqref{111}, $N_{G_i}(X,1)$ is a sum of $2|E_i|(1+o(1))n^{|V_i|-2}$ many terms of the type $X_{l_1}X_{l_2}\cdots X_{l_{|E_i|-1}},$ with distinct edges $l_1,\cdots, l_{|E_i|-1}$.  Thus, using the definition of  $\Psi_{\vec{\beta}}$ from \eqref{phi}, by \eqref{720}, for some constant $C>0$ \begin{align} \label{646}
\Big\vert \mathbb{E} \partial_1 H(X) -  \Psi_{\vec{\beta}}(a_n)\Big\vert \leq  \frac{C}{n}.
\end{align}
Therefore,  applying \eqref{644}, \eqref{645}, and \eqref{646} to \eqref{643},
\begin{align*}
{ \Phi\Big(  \Psi_{\vec{\beta}}(a_n)-\frac{K}{\sqrt{n}} - \frac{C}{n} \Big ) }  -   {2\exp(-cK^2) }  \leq a_n \leq  { \Phi\Big(  \Psi_{\vec{\beta}}(a_n)+\frac{K}{\sqrt{n}} + \frac{C}{n} \Big ) }   + 2\exp(-cK^2).
\end{align*}
 Since one can check that ${\sup_{x\in \R} |\Phi'(x)| }\leq \frac{1}{4}$,
\begin{align} \label{647}  
-\frac{1}{4}\Big(\frac{K}{\sqrt{n}}+\frac{C}{n}\Big) -  { 2\exp(-cK^2) } \leq a_n- { \Phi (  \Psi_{\vec{\beta}}(a_n)  ) }  \leq \frac{1}{4} \Big(\frac{K}{\sqrt{n}} + \frac{C}{n}\Big) + 2\exp(-cK^2).
\end{align}
Recall that $p^*$ satisfies the fixed point equation
\begin{align} \label{648}
p^* =  { \Phi ( \Psi_{\vec{\beta}}(p^*))  =  }  \frac{\exp(\Psi_{\vec{\beta}}(p^*))}{1+\exp(\Psi_{\vec{\beta}}(p^*))}.
\end{align}

Denoting by $f$ the map $t\mapsto t - { \Phi ( \Psi_{\vec{\beta}}(t))}$, note that $f(0)<0$ and $f(1)>0.$
Further, since $\vec\beta$ is in the high temperature regime, the map $f$ has a unique root $p^*$. 
Also recall from \eqref{secondorder} that $f'(p^*)>d$ for some positive constant $d.$ 
Thus, owing to continuity of $f$ on the compact interval $[0,1],$ it follows that for all small enough $\e>0$,  for all $t\in[0,1]$ with $|t-p^*|\ge \e,$ we have $|f(t)|> \frac{d}{2}\e.$

Hence, by \eqref{647}, for all $K$ large enough,
\begin{align*}
|p^*-a_n|\leq \frac{2}{d}\Big(\frac{1}{4} \big(\frac{K}{\sqrt{n}} + \frac{C}{n}\big) + 2\exp(-cK^2)\Big).
\end{align*}
Taking $K=C'\sqrt{\log n}$ with a sufficiently large constant $C'>0$ in the above inequality,  we obtain \eqref{640}.
\end{proof}
Using the above, we now finish the proof of Theorem \ref{w1distance}. 
\begin{proof}[Proof of Theorem \ref{w1distance}]
{ Suppose that  two random variables $X$ and $Y$ are distributed as  ERGM($\vec\beta$)  and $G(n,p^*)$, respectively. In order to prove \eqref{103}, it suffices to construct a coupling of $(X,Y)$ such that
\begin{align} \label{104}
\mathbb{E}(d_H(X,Y)) \leq Cn^{3/2} \sqrt{\log n}.
\end{align} }
Consider two different GD $X(t)$ and $Y(t)$ associated with  ERGM($\vec\beta$) and $G(n,p^*)$, respectively.  Assume that $X(t)$ and $Y(t)$ start from  initial distributions $\pi = $ ERGM($\vec\beta$) and $\nu=G(n,p^*)$ so that each GD  is stationary. Let us couple the initial state $X(0)$ and $Y(0)$ in an arbitrary way. We then couple $X(t)$ and $Y(t)$ inductively in the following natural way: assume that $X(t)$-chain is at $x$ and $Y(t)$-chain is at $y$. Choose a coordinate $I$ uniformly at random, and pick two random variables $Z_1^I$ and $Z_2^I$ that minimizes the total variation distance $d_{TV}(\pi_I(\cdot|x),\nu_I(\cdot |y))$. Then,  $X(t+1)$ is obtained by replacing $x_I$ by $Z_1^I$ and similarly $Y(t+1)$ is the same as $y$ with $y_I$  replaced by $Z_2^I.$ Now note that given $X(t)$ and $Y(t)$, we have,
\begin{align} \label{920}
\mathbb{P}(X_i(t+1)\neq Y_i(t+1)) & \leq \Big(1-\frac{1}{M}\Big)\1\{X_i(t)\neq Y_i(t)\} + \frac{1}{M} d_{TV} (\pi_i(\cdot|X(t)),\nu_i(\cdot |Y(t))).
\end{align}
Since $p^* = { \Phi ( \Psi_{\vec{\beta}}(p^*))  }    $, by the mean value theorem and that ${\sup_{x\in \R} |\Phi'(x)| }\leq \frac{1}{4}$,
\begin{align} \label{927}
d_{TV} (\pi_i(\cdot|X(t)),\nu_i(\cdot |Y(t))) =  |{ \Phi ( \partial_i H(X(t)) )   }   - p^* |  \leq \frac{1}{4}|\partial_i H(X(t)) - \Psi_{\vec\beta}(p^*)|.
\end{align}
Since $X(t)$ is stationary,
by \eqref{646},
\begin{align*}
\Big\vert \mathbb{E}\partial_i H(X(t)) - \sum_{j=1}^s 2\beta_j |E_j| (\mathbb{E} X_1)^{|E_j|-1} \Big\vert  = O\Big(\frac{1}{n}\Big).
\end{align*}
Using this with Proposition \ref{prop 7.3}, we get, 
\begin{align}  \label{922}
\big\vert \mathbb{E}\partial_i H(X(t)) - \Psi_{\vec\beta}(p^*) \big\vert  = O\Big(\frac{\sqrt{\log n}}{\sqrt{n}}\Big).
\end{align}
Since $X(t)$ is  distributed as ERGM($\vec\beta$), recalling \eqref{641}, 
\begin{align} \label{923}
\mathbb{P} \Big(|\partial_i H(X(t)) -  \mathbb{E}\partial_i H(X(t))| > C\frac{\sqrt{\log n} }{\sqrt{n}}\Big)\leq  2 e^{-cn(C^2\log n /n)} =2  n^{-cC^2}.
\end{align}
By  \eqref{922} and \eqref{923}, for all large constant  $C_0>0$, there exists $C_1>c C_0^2/2$ such that
\begin{align} \label{926}
\mathbb{P} \Big(|\partial_i H(X(t)) - \Psi_{\vec\beta}(p^*)|> C_0 \frac{\sqrt{\log n} }{\sqrt{n}}\Big)\leq  n^{-C_1}.
\end{align}
Let us assume that $C_0$ is a sufficiently large constant so that $C_1>3$. By  \eqref{920} and \eqref{927}, given $X(t)$ and $Y(t)$, on the event  $|\partial_i H(X(t)) - \Psi_{\vec\beta}(p^*)|\leq  C_0 \frac{\sqrt{\log n} }{\sqrt{n}}$,
\begin{align} \label{924}
\mathbb{P}(X_i(t+1)\neq Y_i(t+1)) & \leq \Big(1-\frac{1}{M}\Big)\1(X_i(t)\neq Y_i(t)) +  \frac{C_0}{4M} \frac{\sqrt{\log n} }{\sqrt{n}}.
\end{align}
Whereas,  on the event  $|\partial_i H(X(t)) - \Psi_{\vec\beta}(p^*)|>  C_0 \frac{\sqrt{\log n} }{\sqrt{n}}$, we use the trivial bound
\begin{align} \label{925}
\mathbb{P}(X_i(t+1)\neq Y_i(t+1)) & \leq 1.
\end{align}
Thus, using \eqref{926}, \eqref{924}, and \eqref{925},
\begin{align*}
\mathbb{P}(X_i(t+1)\neq Y_i(t+1)) & \leq \Big(1-\frac{1}{M}\Big)\mathbb{P}(X_i(t)\neq Y_i(t)) +  \frac{C_0}{4M}\frac{\sqrt{\log n} }{\sqrt{n}} +n^{-C_1}.
\end{align*}
Recalling that $M=\frac{n(n-1)}{2}$ and $C_1>3$, there exists a constant $C>0$ such that for sufficiently large $n$, for $t\ge c'n^{5/2}$ with $c'>0$, (any $t$ suitably greater than $n^2$ will do),
\begin{align*}
\mathbb{P}(X_i(t)\neq Y_i(t)) \leq C\frac{\sqrt{\log n} }{\sqrt{n}}.
\end{align*}
Adding this over all edges $i$, we obtain
$
\mathbb{E}(d_H(X(t),Y(t))) \leq Cn^{3/2} \sqrt{\log n}.
$ Since $X(t)$ and $Y(t)$ are  distributed as ERGM($\vec\beta$) and $G(n,p^*)$ respectively,  we obtain \eqref{103}.
\end{proof}

{
\begin{remark}
As the reader might notice, in the above proof,
instead of choosing edges uniformly to update, one can also perform a  sequential update of all the edges. More precisely, ordering the edges in an arbitrary fashion labelling them $1,\ldots M$, let $(X(t))_{t=0,1,\cdots,M}$ be a Markov chain starting from the initial distribution  $X(0)=$ERGM($\vec\beta$) such that  given $X(t-1)$, $X(t)$ is obtained by resampling the $t^{th}$ edge, according to the  conditional distribution given other edges. Similarly, define the corresponding Markov chain $Y(t)_{t=0,1,\cdots,M}$ for $G(n,p^*)$.  It is obvious that both $X(t)$ and $Y(t)$ are stationary. Now consider the natural coupling of the two Markov chains by taking an arbitrary coupling $(X(0),Y(0))$ and at each time $t\geq 1$, given $(X(t-1), Y(t-1)),$ taking an optimal coupling for $X_t(t)$ and $Y_t(t)$  that attains the total variation distance. Then, by \eqref{927} and \eqref{926},
\begin{align*}
\mathbb{P}(X_{t+1}(t+1)\neq Y_{t+1}(t+1)) =d_{TV} (\pi_t(\cdot|X(t)),\nu_t(\cdot |Y(t))) \leq \frac{C_0}{4}\frac{\sqrt{\log n} }{\sqrt{n}} +n^{-C_1}.
\end{align*}
Adding this over $t=1,2,\cdots,M$, since $C_1>3$, we have
\begin{align*}
\mathbb{E}(d_H(X(M),Y(M))) &= \sum_{t=1}^M \mathbb{P}(X_t(M)\neq Y_t(M))  \\
&= \sum_{t=1}^M \mathbb{P}(X_t(t)\neq Y_t(t))  \leq  M\Big(\frac{C_0}{4}\frac{\sqrt{\log n} }{\sqrt{n}} +n^{-C_1}\Big) = O(n^{3/2}\sqrt{\log n}).
\end{align*}
Since two chains $X(t)$ and $Y(t)$ are stationary, this concludes the proof.
\end{remark}
}

\subsection{Pseudo-likelihood estimators}\label{pseudo-section}Being a central object in the study of statistical models on networks, a natural question about the ERGM, is whether one can deduce any meaningful estimate of the parameter $\vec{\beta}$ given a single realization of the graph. 
A well known estimator in such contexts is the so-called Maximum Likelihood Estimator (MLE), which as the name suggests, are the values of the parameter that maximizes the probability of the realized sample. However often in high dimensional models such as the ERGM, the optimization problem associated to finding the MLE, neither has a closed form solution, nor is computationally feasible. In such settings a related proxy is the notion of a \it{pseudo-likelihood estimator}. This was was  first introduced by Besag \cite{b1,b2} in the context of analyzing data possessing spatial dependence. To define this precisely, consider a random vector $X=(X_1,\cdots,X_n)$ with probability density function $f(\beta,X),$ parametrized by a parameter $\beta$. Define $f_i(\beta,x)$ to be a conditional probability density of $X_i$ given $(X_j)_{j\neq i}$. The \it{maximum pseudo-likelihood estimator} (MPLE) of $\beta$ is defined by
\begin{align*}
\hat{\beta}_{\text{MPLE}}:=\argmax_\beta \prod_{i=1}^n f_i(\beta,X).
\end{align*}
In practise, this is often much simpler to analyze than the MLE problem,
\begin{align*}
\hat{\beta}_{\text{MLE}}:=\argmax_\beta f(\beta,X).
\end{align*}
A sequence $\{\hat{\beta}_{n}\}_n$ is said to be a \it{consistent} sequence of estimators for $\beta$ provided that for any $\varepsilon>0$,
\begin{align*}
\lim_{n\rightarrow \infty} \mathbb{P}(|\hat{\beta}_n - \beta| > \varepsilon) = 0.
\end{align*}

While studying the consistency of MLE in rather general settings has been an important classical theme in probability theory and statistics, more recently, there has been a parallel interest in understanding the MPLEs for various statistical mechanics models, in particular for spin systems. One of the most notable results was proved by Chatterjee \cite{c3} where among various things, it was proved  that the MPLE in the
Sherrington-Kirkpatric (S-K) model \cite{sk} and the Hopfield model \cite{h} with a single parameter $\beta>0$ is consistent  for $\beta$, whereas, consistency does not hold in the high temperature 
Curie-Weiss model. In fact,  if the inverse temperature $\beta$ satisfies $0\leq \beta<1$, then $\hat{\beta}_n \rightarrow 1$ in probability (see \cite[Section 1.7]{c3} for details).
Later, similar problems for the 
Ising model  on regular graphs and on Erd\H{o}s-R\'enyi graphs were studied in \cite{bm} and more recently in \cite{gm}.

While such a study for the ERGM has not yet appeared in the  literature, the consistency of the MPLE for the ERGM with a single parameter $\beta$ in the whole positive temperature regime $\beta>0$ is expected to follow using the robust  arguments developed in \cite{c3}. 
On the other hand, it seems that a general theory of the consistency of MPLE with several parameters has not been developed yet, except in some special cases (see {\cite{gm,ddp}}). 
In fact,
the problem seems to be ill-posed in the case of the ERGM, since, in the high temperature case,  {using the concentration result Theorem \ref{Theorem 2}, with high probability, the critical equation corresponding to finding $\hat{\vec{\beta}}$, the MPLE for $\vec\beta$,
seems to only say that $p^* \approx \Psi_{\hat{\vec\beta}}(p^*),$ and hence does not provide any further information about the vector $\vec \beta.$
This might lead one to wonder for any parameters $\vec\beta$ and $\vec \gamma$ with $\varphi_{\vec\beta}(p^*)=p^*=\varphi_{\vec\gamma}(p^*)$, since both ERGM($\vec\beta$) and  ERGM($\vec\gamma$) behave like $G(n,p^*)$, if they are contiguous. 
Nonetheless, while in light of the above, the MPLE is unable to estimate the parameters, as was mentioned earlier, in \cite{mukherjee1} a consistent estimator was obtained for the parameters in the special case of the two-star model, ruling out the possibility of such a contiguity result, at least in this case.

\bibliography{exponential}
\bibliographystyle{plain}

\end{document}